\documentclass[letterpaper,reqno,11pt,oneside]{amsart} 
\pdfoutput=1

\usepackage{amsmath,amsthm,amsfonts,amssymb}
\usepackage[extension=pdf]{hyperref}
\usepackage{enumitem,comment}
\usepackage[T1]{fontenc}
\usepackage{times}
\usepackage{soul}
\usepackage{mathtools}
\usepackage{mathrsfs}
\usepackage{cases}
\usepackage{needspace}
\usepackage{accents}
\usepackage{bm}
\usepackage[scr=boondoxo,scrscaled=1.05]{mathalfa}
\usepackage[dvipsnames]{xcolor}
\usepackage[multiple]{footmisc}
\usepackage{microtype}
\usepackage{import}
\usepackage{xspace}
\usepackage{tikz}
\usetikzlibrary{matrix,backgrounds,decorations.pathreplacing}
\usepackage{rotating}
\usepackage[totalheight=9.6in,totalwidth=6.15in,centering]{geometry}
\usepackage{graphics,graphicx}
\usepackage[style=alphabetic,natbib=true,maxnames=99,isbn=false,doi=false,url=false,firstinits=true,hyperref=auto,arxiv=abs,backend=bibtex]{biblatex}
\AtEveryBibitem{%
\clearlist{language}}
\DeclareFieldFormat[article,inbook,incollection,inproceedings,patent,thesis,unpublished]{title}{#1\isdot}
\renewbibmacro{in:}{%
\ifentrytype{article}{}{%
\printtext{\bibstring{in}\intitlepunct}}} 
\defbibheading{apa}[\refname]{\section*{#1}}
\DeclareFieldFormat{sentencecase}{\MakeSentenceCase{#1}} \renewbibmacro*{title}{%
\ifthenelse{\iffieldundef{title}\AND\iffieldundef{subtitle}}{}
{\ifthenelse{\ifentrytype{article}\OR\ifentrytype{inbook}%
    \OR\ifentrytype{incollection}\OR\ifentrytype{inproceedings}%
    \OR\ifentrytype{inreference}} {\printtext[title]{%
      \printfield[sentencecase]{title}%
      \setunit{\subtitlepunct}%
      \printfield[sentencecase]{subtitle}}}%
  {\printtext[title]{%
      \printfield[titlecase]{title}%
      \setunit{\subtitlepunct}%
      \printfield[titlecase]{subtitle}}}%
  \newunit}%
\printfield{titleaddon}}

\usepackage[f]{esvect}
\usepackage{stmaryrd}

\DeclareFontFamily{U}{BOONDOX-calo}{\skewchar\font=45 }
\DeclareFontShape{U}{BOONDOX-calo}{m}{n}{
  <-> s*[1.05] BOONDOX-r-calo}{}
\DeclareFontShape{U}{BOONDOX-calo}{b}{n}{
  <-> s*[1.05] BOONDOX-b-calo}{}
\DeclareMathAlphabet{\mcb}{U}{BOONDOX-calo}{m}{n}
\SetMathAlphabet{\mcb}{bold}{U}{BOONDOX-calo}{b}{n}

\makeatletter
\newcommand{\sbullet}{%
  \hbox{\fontfamily{lmr}\fontsize{.6\dimexpr(\f@size pt)}{0}\selectfont\textbullet}}

\makeatother

\makeatletter
\newcommand*\bigcdot{\mathpalette\bigcdot@{.5}}
\newcommand*\bigcdot@[2]{\mathbin{\vcenter{\hbox{\scalebox{#2}{$\m@th#1\bullet$}}}}}
\makeatother

\definecolor{darkblue}{rgb}{0.13,0.13,0.39}%
\hypersetup{colorlinks=true,urlcolor=darkblue,citecolor=darkblue,linkcolor=darkblue,%
  pdftitle=Myopic non-intersection in a periodic potential,pdfauthor={J. Arista, A. Sepúlveda, D. Remenik}}

\mathtoolsset{showmanualtags,showonlyrefs}

\newtheorem{thm}{Theorem}[section] 
\newtheorem{lem}[thm]{Lemma}
 
\newtheorem{prop}[thm]{Proposition}

\newtheorem{claim}[thm]{Claim}
\theoremstyle{definition} 
\newtheorem{rem}[thm]{Remark} 
\newtheorem{defn}[thm]{Definition}

\newcounter{algoc}

\newtheorem{algo}[algoc]{Algorithm}

\renewcommand{\d}{d}%

\newcommand{\V}[1]{\mathbb{V}_{\hspace{-0.1em}#1}}

\newcommand{\B}{B}

\newcommand{\D}{\mathbb{D}}

\renewcommand{\H}{\mathbf{F}}

\newcommand{\CM}{\mathcal{M}}

\newcommand{\cL}{\mathcal{L}}

\newcommand{\caja}[1]{\left [ \tts #1\right ]\mathchoice{\!}{\tsm}{}{}\ttsm}
\newcommand{\cajab}[1]{\big[ \tts #1\big ]\ttsm}

\newcommand{\x}{\mathsf{x}}
\newcommand{\X}{\mathsf{X}}
\newcommand{\Z}{\mathsf{Z}}
\newcommand{\Y}{\mathsf{Y}}
\newcommand{\wcX}{\widecheck{\X}}
\newcommand{\wcY}{\widecheck{\Y}}

\newcommand{\I}{{\rm i}} 
\newcommand{\pp}{\mathbb{P}}

\newcommand{\ee}{\mathbb{E}} 
\newcommand{\rr}{\mathbb{R}}
\newcommand{\nn}{\mathbb{N}} 
\newcommand{\zz}{\mathbb{Z}}

\newcommand{\M}{\CM}

\newcommand{\p}{\partial}
\newcommand{\uno}[1]{\mathbf{1}_{#1}}
\newcommand{\ep}{\varepsilon}
\newcommand{\eps}{\varepsilon}
\newcommand{\vs}{\vspace{6pt}}
\newcommand{\wt}{\widetilde}

\newcommand{\m}{\overline{m}}

\newcommand{\qqand}{\qquad\text{and}\qquad}

\newcommand{\ts}{\hspace{0.1em}}
\newcommand{\tts}{\hspace{0.05em}}
\newcommand{\tsm}{\hspace{-0.1em}}
\newcommand{\ttsm}{\hspace{-0.05em}}

\makeatletter
\newcommand\RedeclareMathOperator{%
  \@ifstar{\def\rmo@s{m}\rmo@redeclare}{\def\rmo@s{o}\rmo@redeclare}%
}
\newcommand\rmo@redeclare[2]{%
  \begingroup \escapechar\m@ne\xdef\@gtempa{{\string#1}}\endgroup
  \expandafter\@ifundefined\@gtempa
     {\@latex@error{\noexpand#1undefined}\@ehc}%
     \relax
  \expandafter\rmo@declmathop\rmo@s{#1}{#2}}
\newcommand\rmo@declmathop[3]{%
  \DeclareRobustCommand{#2}{\qopname\newmcodes@#1{#3}}%
}
\@onlypreamble\RedeclareMathOperator
\makeatother
\newcommand{\uptext}[1]{\text{\upshape{#1}}}

\RedeclareMathOperator{\det}{\mathop{\uptext{det}}}
\RedeclareMathOperator{\ker}{\mathop{\uptext{ker}}}
\RedeclareMathOperator{\exp}{\mathop{\uptext{exp}}}
\RedeclareMathOperator{\log}{\mathop{\uptext{log}}}
\RedeclareMathOperator*{\lim}{\mathop{\uptext{lim}}}
\RedeclareMathOperator*{\sup}{\mathop{\uptext{sup}}}
\RedeclareMathOperator*{\limsup}{\mathop{\uptext{lim\hspace{1pt}sup}}}
\RedeclareMathOperator*{\liminf}{\mathop{\uptext{lim\hspace{1pt}inf}}}
\RedeclareMathOperator*{\max}{\mathop{\uptext{max}}}
\RedeclareMathOperator*{\inf}{\mathop{\uptext{inf}}}
\RedeclareMathOperator*{\min}{\mathop{\uptext{min}}}
\RedeclareMathOperator*{\cos}{\mathop{\uptext{cos}}}
\RedeclareMathOperator*{\sin}{\mathop{\uptext{sin}}}
\RedeclareMathOperator*{\arcsin}{\mathop{\uptext{arcsin}}}
\RedeclareMathOperator*{\arg}{\mathop{\uptext{arg}}}
\RedeclareMathOperator*{\mod}{\mathop{~~~\uptext{mod}}~~}
\RedeclareMathOperator{\Re}{\mathop{\uptext{Re}}}
\RedeclareMathOperator{\Im}{\mathop{\uptext{Im}}}

\newcommand{\twopii}[1]{\ifthenelse{#1=1}{2\pi\I}{(2\pi\I)^{#1}}}

\newcommand{\tl}{\tilde{\ell}}

\DeclareFontFamily{U}{mathx}{}
\DeclareFontShape{U}{mathx}{m}{n}{<-> mathx10}{}
\DeclareSymbolFont{mathx}{U}{mathx}{m}{n}
\DeclareMathAccent{\widehat}{0}{mathx}{"70}
\DeclareMathAccent{\widecheck}{0}{mathx}{"71}

\def\dash---{\kern.16667em---\penalty\exhyphenpenalty\hskip.16667em\relax}

\numberwithin{equation}{section}

\let\oldmarginpar\marginpar
\renewcommand\marginpar[1]{\-\oldmarginpar[\raggedleft\footnotesize #1]%
  {\raggedright{\small\textsf{#1}}}}

\allowdisplaybreaks[2]

\begin{document}
\title{Myopic non-intersection in a periodic potential}
\author{Jonas Arista}
\address[J.~Arista]{
  Independent researcher} \email{jonas.arr@gmail.com}
\author{Daniel Remenik}
\address[D.~Remenik]{
  Departamento de Ingenier\'ia Matem\'atica and Centro de Modelamiento Matem\'atico (IRL-CNRS 2807)\\
  Universidad de Chile\\
  Av. Beauchef 851, Torre Norte, Piso 5\\
  Santiago\\
  Chile} \email{dremenik@dim.uchile.cl}
\author{Avelio Sepúlveda}
 \address[A.~Sepúlveda]{
  Departamento de Ingenier\'ia Matem\'atica and Centro de Modelamiento Matem\'atico (IRL-CNRS 2807)\\
  Universidad de Chile\\
  Av. Beauchef 851, Torre Norte, Piso 5\\
  Santiago\\
  Chile} \email{lsepulveda@dim.uchile.cl}
\date{June 2025}

\begin{abstract}
    We introduce a class of Markov processes conditioned to avoid intersection over a moving time window of length $T > 0$, a setting we refer to as \emph{myopic non-intersection}. In particular, we study a system of myopic non-intersecting Brownian motions subject to a periodic potential. Our focus lies in understanding the interplay between the confining effect of the potential and the repulsion induced by the non-intersection constraint. We show that, in the long time limit, and as both $T$ and the strength of the potential become large, the model converges to a system of myopic non-intersecting random walks, which transitions between standard non-intersection dynamics and exclusion behavior. The main technical contribution of the paper is the introduction of an algorithm, based on a modification of the acceptance-rejection sampling scheme, that provides an explicit construction of myopically constrained systems.
\end{abstract}

\maketitle

\section{Introduction and main results}

\subsection{Background and motivation}

Systems of non-intersecting random paths in one dimension have been studied intensively for more than two decades.
The canonical example is the model of \emph{non-intersecting Brownian motions}, which corresponds to a collection of $N$ Brownian motions conditioned (suitably, in Doob's sense) on the event that they never intersect.
Dyson observed in \cite{dyson} that this process has the same distribution as the evolution of the eigenvalues of a matrix evolving as a Brownian motion on the space of $N\times N$ Hermitian matrices with a suitably chosen initial condition, a process known as Dyson Brownian motion, and which plays a key role in random matrix theory.

Classical systems of non-intersecting paths such as Dyson Brownian motion are \emph{determinantal}---their correlation functions take a specific form in terms of determinants---which has allowed the derivation of explicit formulas for many quantities of interest.
This fact is intimately connected to the classical Karlin--McGregor formula, which expresses the transition probabilities of $N$ independent Brownian motions killed upon intersection as a simple determinant, and to the characterization of many non-intersecting systems as the Doob $h$-transform of the independent system with $h$ given as a Vandermonde determinant. We refer the reader to \cite{grabiner,konigOConnellRoch,eichelsbacherKonig} for more details on these connections.

Beyond their intrinsic interest, non-intersecting paths are also a powerful tool in integrable probability, particularly in the study of models in the KPZ universality class. Many such models can be recast in terms of particular (often discrete) systems of non-intersecting paths, which can then be studied using techniques from the theory of determinantal processes. For instance, the boundary of the frozen region in domino tilings of the Aztec diamond can be mapped to the top path of a system of non-intersecting random walks which encodes the tiling.
This mapping has been used to show that the rescaled boundary converges to the Airy$_2$ process \cite{johanssonArctic}.
Another example is the totally asymmetric simple exclusion process (TASEP), which for the special case of step initial condition can be coupled with another system of non-intersecting random walks, see \cite{borodinFerrariTilings} (and also \cite{warren} for a related construction in the case of non-intersecting Brownian motions).
This coupling can be used to give a relatively simple proof that, after proper rescaling, the one-point distribution of the process converges to a Tracy--Widom GUE random variable \cite{tracyWidomGUE}.
Furthermore for TASEP with general initial conditions, the explicit solution derived in \cite{fixedpt} has also been recast, more recently in \cite{bisiLiaoSaenzZygouras}, in terms of systems of non-intersecting paths.

In this paper, we explore a variation of the standard setting for non-intersecting Brownian motions from two different angles. First, we subject the system to a periodic potential, with the goal of understanding  how its confining effect interacts with the intrinsic repulsion arising from the non-intersection constraint. However, as we will see, it turns out to be more interesting to do this under a second variation: instead of conditioning on the paths never intersecting, we consider systems with \emph{myopic non-intersection}, where paths are dynamically conditioned to not intersect over a moving time window of length $T>0$.
The parameter $T$ has the effect of modulating the strength of the repulsion, and this effect will allow us to observe interesting behavior in the long time limit as both $T$ and the strength of the potential become large.
To place the system in a setting closer to KPZ models such as TASEP, we introduce asymmetry by adding a drift to the potential.

The original motivation for this work lies in studying the interplay between repulsion and confinement in the model, and this remains our primary focus. However, the study of myopic non-intersection is interesting in its own right and, to the best of our knowledge, has not been undertaken in the literature. 
While we leave a more detailed exploration of it for future work, we lay some groundwork by establishing basic properties of these processes and by presenting two different constructions of such models: one via a limiting procedure, and another through an explicit algorithmic construction.
We stress that, due to the myopic nature of the non-intersection in our models, these constructions necessarily have to go beyond the standard determinantal or Doob $h$-transform methods.

\subsection{The model}

\subsubsection{Non-intersecting Brownian motions in a periodic potential}\label{s.mbm-intro}

Consider a periodic function $u\in C^\infty(\rr)$, with period equal to $1$, which we think of as being prescribed on $[-\frac12,\frac12)$.
This function $u$ represents the basic periodic potential underlying our model.
However, since we are interested in Brownian motions with drift component, it is more convenient to absorb that drift into the potential itself. To that end, for a fixed parameter $b>0$, we define a modified potential
\begin{align*}
v(x) = u(x)-bx.
\end{align*}
We assume, in addition, that the resulting potential $v$ has\footnote{%
To fix ideas one can think of the choice $u(x)=\sin(\pi x)^2=\frac12(1-\cos(2\pi x))$. 
However, while this function has the prescribed local minima and maxima at integers and half-integers, respectively, this is not the case for $v$ as the drift term changes the critical points. 
We could account for this in our results without any essential difficulty, but in order to keep the notation simpler we make the assumption that $v$ is adjusted so that the critical points are located as specified. 
For the trigonometric choice of $u$ which we suggested, the definition could be adjusted, for small $b$, to
\[v(x)=\frac{2+b}{2}\sin(\pi x)^2+\frac{b}{4\pi}\sin(4\pi x)-bx.\]}
a unique local minimum at zero in $[-\frac12,\frac12]$ with $v(0)=0$ and a unique local maximum at $\frac 12$ in $[0,1]$ with $v(\frac 1 2)=1$. 

Fix $\kappa>0$, and define $\X$ as the solution of the SDE
\begin{align}\label{e.SDE}
\d\X(t)= -\kappa v'(\X(t)) \d t  + \d\B_t,
\end{align}
where $B$ is a (one-dimensional) Brownian motion. The model we study is built out of this diffusion. The parameter $\kappa>0$ parametrizes the height difference between the consecutive local maxima and minima of $v$ and should be thought of as being large.

\begin{figure}
    \centering
    \includegraphics[width=0.5\linewidth]{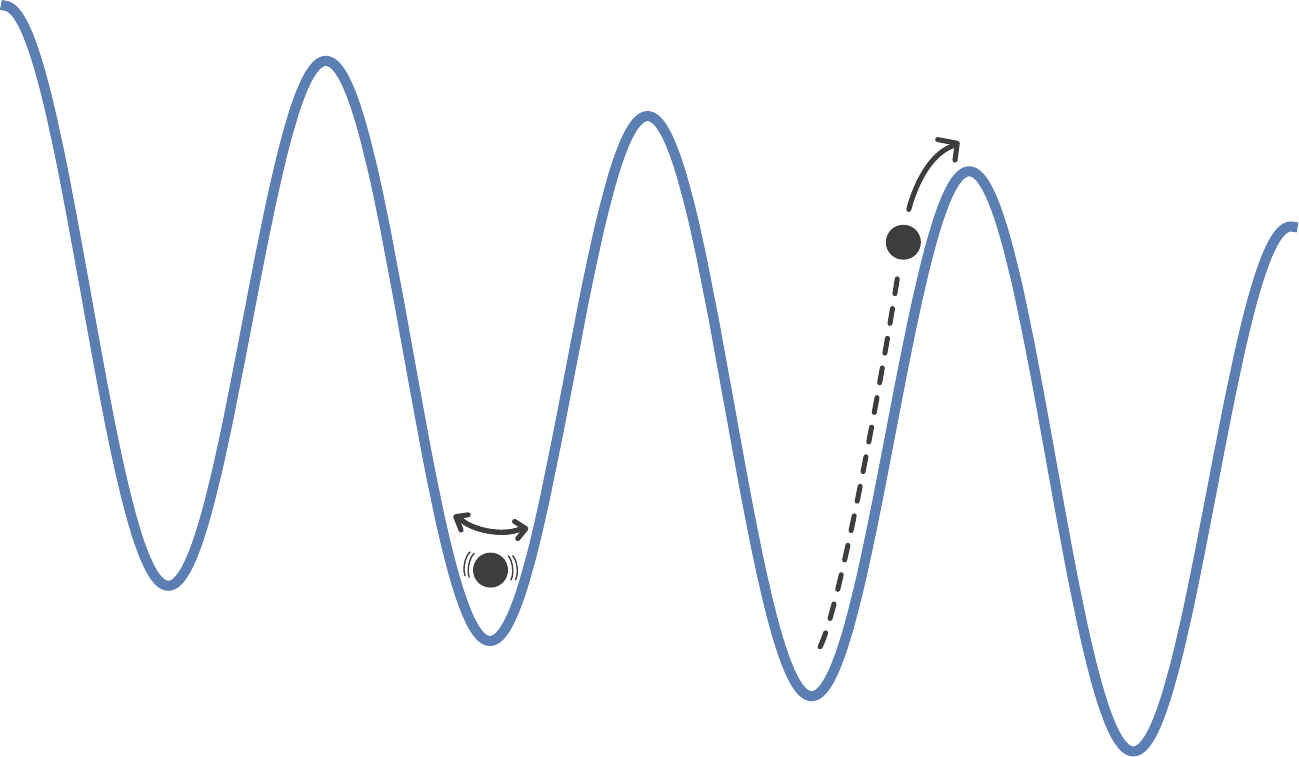}
    \caption{A Brownian particle in a periodic potential with drift. The particle tends to be stuck near a minimum of the potential for long periods of time. Occasionally the particle overcomes the potential and jumps to a neighboring box which, due to the drift, is overwhelmingly more likely to be the one to its right.}
    \label{f.periodic-potential}
\end{figure}

For large $\kappa$, the above choice of potential splits the real line into countably many boxes of length $1$ centered at the integers, so that a single Brownian particle subjected to it tends to spend most of its time stuck inside a box (near the corresponding minimum of $v$), with transitions to a neighboring box which are relatively rare.
Moreover, the additional drift term makes it much more likely that the particle jumps to the right rather than to the left.
See Figure \ref{f.periodic-potential}.
This intuition can be made precise in this setting through standard metastability results for diffusions with small noise.
Let $\caja{\X}(t)$ denote the last integer that $\X$ passed before time $t$ or, more explicitly,
\begin{equation}
  \caja{\X}(0)=\lfloor \X(0) + 1/2\rfloor\qqand \caja{\X}(t)=\X_{\tau(t)}\quad\uptext{for}\quad t>0,\label{e.caja}
\end{equation}
where
\[\tau(t)=\sup\big\{s\leq t: \X_s \in \zz\}\ \text{ and } \ \X(-\infty)= \caja{\X}(0).\]
Then there exist constants $\lambda^\kappa=e^{2\kappa+o(\kappa)}$ so that
\begin{equation}\label{e.FW}
    \caja{\X}(\lambda^\kappa t)\xrightarrow[\kappa\to\infty]{}\Y(t)
\end{equation}
in distribution, where $\Y(t)$ is a totally asymmetric simple random walk on the integers with jump rate $1$ (or in other words, a Poisson process with rate $1$).
We prove this in Proposition \ref{p.FW}, by alluding to classical results on metastability for Brownian motion in a double well potential, which we present in Section \ref{sec:F-W}.

In order to introduce the main motivation for this work, consider now a system of $N$ copies $(\X_1(t),\dotsc,\X_N(t))$ of the diffusion \eqref{e.SDE} which we think of, informally, as being conditioned on never intersecting\footnote{
Constructing such a system is, of course, delicate, because the conditioning is singular. As far as we know, there is no available construction in the literature for general potentials. But in this paper we will actually not work directly with this system, so we keep the discussion at an informal level.}.
We aim at understanding the interplay, as $\kappa$ gets large, between two competing effects in this model: the overall non-intersection condition, which boils down to a long-range repulsion interaction; and a strong (periodic) confining potential, which weakens that interaction as independent particles that are not in the same box are, for a long time, unlikely to intersect.

Specifically, we are interested in the behavior of the conditioned process as $\kappa \to \infty$ under the time scaling introduced above.
In view of \eqref{e.FW} it is reasonable to expect that for the sped up system one has
\begin{equation}
    (\caja{\X_1}(\lambda^\kappa t),\dotsc,\caja{\X_N}(\lambda^\kappa t))\xrightarrow[\kappa\to\infty]{}(\overline\Y_1(t),\dotsc,\overline\Y_N(t))
\end{equation}
in distribution for some limiting process $(\overline\Y_1(t),\dotsc,\overline\Y_N(t))$ taking values in $\zz^N$. 
However (and perhaps not too surprisingly), it turns out that the effect of conditioning on particles \emph{never} intersecting is too strong, as it affects all time scales simultaneously, making the second effect disappear.
In fact, two independent particles occupying different boxes will tend to stay separated for a long time if $\kappa$ is large, but after conditioning on the event that the particles never intersect, they still strongly influence each other as they are likely to hit at scale $\lambda^\kappa$. This suggests that, as $\kappa\to\infty$, they will behave exactly as a Poisson process conditioned on never intersecting. We do not prove this directly in the paper, as it would divert us from our main focus, but we do provide strong evidence for it; see Remark \ref{r.non-myopic}.

The above suggests that a more interesting behavior may be observed if we weaken the non-intersection condition.

\subsubsection{Myopic non-intersection}\label{s.myopic-intro}

We introduce now the main object of study of this paper, corresponding to a system $\X^{(T)}=(\X^{(T)}_{1},\dotsc,\X^{(T)}_N)$ of $N$ copies of the diffusion \eqref{e.SDE} whose evolution is infinitesimally conditioned at every time to not intersect over the next $T$ units of time, where $T$ is a positive parameter.

One way to think of such a process is as follows.
Let $\wcX^{(T)}=(\wcX^{(T)}_{1},\dotsc,\wcX^{(T)}_N)$ be a system of $N$ copies of the diffusion \eqref{e.SDE} which are  conditioned on the event that they do not intersect up to time $T$ (and are independent after time $T$).
This defines a time-inhomogeneous Markov process; let $\mathcal L_t$ denote its infinitesimal generator at time $t$.
Then the process $\X^{(T)}=(\X^{(T)}_{1},\dotsc,\X^{(T)}_N)$ can be thought of as the time-homogeneous Markov process with generator $\cL_0$.

An alternative way to construct the process, and the one which we choose in this paper, is by approximation: we define it by concatenating an order of $1/\ep$ non-intersecting systems $\wcX^{(T)}$ run up to time $\ep$, and then taking $\ep\to0$.
We show in Section \ref{s.mBM} that this limit exists, and defines a continuous Markov process made out of non-intersecting paths; we will call this process a system of \emph{myopic non-intersecting Brownians motion in a periodic potential}, or \emph{mBM} for short, and refer to the parameter $T$ as the \emph{foresight} of the system. Although it should be possible to show that the process arising in the limit coincides with the one which could be defined through the generator approach described above, the construction by approximation is more intuitive and better suited for our methods, so it is the one we use.

The main tool which we introduce in the paper, and on which the proof of our main result is based, is a third construction of the mBM, based on a certain generalization of the acceptance-rejection sampling method, see Section \ref{sec:direc-mBM}.
This construction has the advantage that it involves no limiting procedure, and is very well suited to coupling techniques. In fact, we use a coupling argument to show that the algorithm coincides with the limiting procedure. We regard this acceptance-rejection sampling method for the mBM as one of the main technical contributions of this work.

\subsection{Main results}

The main result of this work is that, as $\kappa \to \infty$, the mBM in a periodic potential converges in law to a process which interpolates between TASEP with $N$ particles and a system of $N$ Poisson random walks conditioned on never intersecting.

Let us denote the limiting system as $\Y^{(L)}=(\Y^{(L)}_1, \Y^{(L)}_2, \dotsc, \Y^{(L)}_N)$, and we refer to it as a system of \emph{myopic non-intersecting random walks} (or \emph{mRW} for short) with foresight $L$.
It corresponds to a system of $N$ copies of a rate $1$ Poisson process which, as for the mBM, is dynamically conditioned at every time to not intersect for the next $L$ units of time.
Its explicit construction is provided in Section \ref{s.mRW} (based on a direct construction of its generator).

Our first results shows that $\Y^{(L)}$ indeed interpolates between TASEP and non-intersecting Poisson random walks.
Before stating it, let us briefly introduce the two limiting processes.

The \emph{totally asymmetric simple exclusion process (TASEP)} (with finitely many particles) consists of $N$ particles on $\zz$ with positions\footnote{This is the opposite of the usual ordering convention for the TASEP particles, but it fits the one we are employing for non-intersecting paths.} $Y_t(1)<\dotsm<Y_t(N)$ which evolve in continuous time as follows: each particle independently attempts to jump to its neighbor to the right at rate $1$, but jumps are allowed only if that site is unoccupied.
We think of this particle system as a Markov process taking in values in the \emph{Weyl chamber}
\begin{align*}
\Omega_N=\{x \in \rr^N\!: x_1<x_2<\dotsm<x_N\} \;\;\; \text{or, equivalently for this process, } \  \Omega_N^\zz = \Omega_N\cap \zz^N.
\end{align*}
If particles jump at rate $1$ but choose their target site to be its neighbor to the right with probability $p$ and the one to its left with probability $1-p$ (with the same exclusion restriction), the process is known as the \emph{asymmetric simple exclusion process (ASEP)}.

A system of \emph{non-intersecting Poisson random walks} consists of a collection of $N$ independent Poisson processes $Y_1,\dotsc,Y_N$ conditioned on the event
\[\big\{(Y_1(t),\dotsc,Y_N(t))\in\Omega_N^\zz,\quad\forall\,t\geq0\big\}.\]
Since this event has zero probability, this definition needs some care.
The standard way to define the process is through a Doob $h$-transform for $h$ chosen as the so-called \emph{Vandermonde determinant} $\Delta(y)=\prod_{1\leq i<j\leq N}(y_j-y_i)$, which is harmonic for the system of independent walks and vanishes on $\Omega_N^\zz$.
It leads to defining the system as the Markov process on $\Omega_N^\zz$ with generator
\begin{equation}
    \cL^{(\infty)}F(y)=\sum_{i=1}^N\big(F(y+e_i)-F(y))\frac{\Delta(y+e_i)}{\Delta(y)},\label{e.generator_niRW}
\end{equation}
where $e_i$ denotes the $i$-th canonical vector.
For more details, see e.g. \cite{konigOConnellRoch}.

\begin{thm}\label{t.interpolation}
    Consider a system of myopic non-intersecting random walks $\Y^{(L)}$ with foresight $L$.
    Then the limits 
    \[\Y^{(0^+)}=\lim_{L\searrow0}\Y^{(L)}\qqand\Y^{(\infty)}=\lim_{L\nearrow\infty}\Y^{(L)}\]
    exist in distribution with respect to Skorohod topology on compact sets.
    Furthermore:
    \begin{itemize}
        \item $\Y^{(0^+)}$ has the law of TASEP with $N$ particles.
        \item $\Y^{(\infty)}$ has the law of a system of $N$ Poisson random walks conditioned on never intersecting.
    \end{itemize}
\end{thm}

The fact that $\Y^{(L)}$ converges as $L\nearrow\infty$ to non-intersecting random walks follows almost by definition.
In the opposite limit $L\searrow 0$, the conditioning becomes effectively local in time: the process only avoids immediate intersections. This results in a dynamics where attempted jumps onto occupied sites are blocked—--exactly the exclusion rule in TASEP.

In the case when all particles start at the origin, $Y^{(\infty)}$ is the process version of the \emph{Charlier ensemble} from random matrix theory (see e.g. \cite{konigOConnellRoch}).
It is a determinantal process whose correlation kernel is given in terms of Charlier functions, and can be thought of as a discrete analogue of \emph{Dyson Brownian motion} which describes non-intersecting Brownian motions with the same initial condition.
In that sense, $\Y^{(L)}$ can be regarded as providing an interpolation between KPZ and random matrix objects. This complements other types of intermediate structures connecting KPZ models and random matrices, such as certain Markov dynamics on interlacing structures like Gelfand-Tsetlin patterns, which couple specific KPZ particle systems with certain random matrix ensembles (see e.g. \cite{ferraryWhy} and references therein).

We can finally state our main result.
In this statement (and throughout the rest of the paper) we extend the notation $\caja{\X}$ to an mBM $\X^{(T)}=(\X_1^{(T)},\dotsc,\X_N^{(T)})$ coordinatewise:
\begin{equation}\label{e.caja-vec}
  \cajab{\X^{(T)}}=\big(\cajab{\X_1^{(T)}},\dotsc,\cajab{\X_N^{(T)}}\big).
\end{equation}

\begin{thm}\label{t.main}
    Let $\X^{(T^\kappa)}$ be an mBM in a periodic potential with foresight $T^\kappa$ for the equation \eqref{e.SDE} with parameter $\kappa$ and starting from a point $x\in \Omega_N$ such that $x_i\notin\zz+\frac12$ for each $i$.
    Assume that $\frac{T^\kappa}{\lambda^k}\xrightarrow[\kappa\to\infty]{}L>0$, with $\lambda^\kappa=e^{2\kappa+o(\kappa)}$ as in \eqref{e.FW}.
    Then
    \begin{align*}
    \cajab{\X^{(T^\kappa)}}(\lambda^\kappa \cdot) \xrightarrow[\kappa\to\infty]{} \Y^{(L)}(\cdot)
    \end{align*}
    in distribution, with respect to the local Skohorod topology on càdlàg functions, where $\Y^{(L)}$ is the mRW with foresight $L$.
\end{thm}

In words, if the foresight $T^\kappa$  grows like $\lambda^\kappa L$ and time is rescaled at the same scale $\lambda^\kappa$, then the process of integer sites mBM visited by the mBM, $\caja{\X^{(T^\kappa)}}$, converges to an mRW $\Y^{(L)}$. This process, in turn, interpolates between TASEP (as $L\searrow0$) and non-intersecting Poisson random walks (as $L\nearrow\infty$).

\begin{rem}\label{r.non-myopic}
  As a consequence of Theorems \ref{t.interpolation} and \ref{t.main}, and in the setting of the second result, we have that, as $\kappa\to\infty$ and then $L\to\infty$, $\caja{\X^{(T^\kappa)}}(\lambda^\kappa \cdot)$ converges in law to the system $\Y^{(\infty)}$ of non-intersecting Poisson processes.
  This justifies our conjecture that if the system of diffusions is conditioned on never intersecting, then it does not feel the fact that the confinement makes it easy for particles in different boxes to stay apart for a long time.
\end{rem}

\section{Diffusions conditioned on non-intersection in a finite interval}\label{s.diffusions}

The basic object underlying our construction of the mBM is the process $\wcX^{(T)}$, introduced in Section \ref{s.myopic-intro}.
This is a system of $N$ copies of the diffusion \eqref{e.SDE}, conditioned on the event that they do not intersect up to time $T$.
If the system starts at $x\in\Omega_N$, this event has positive probability and $\wcX^{(T)}$ can be constructed directly by conditioning, but for $x\in\partial\Omega_N$ (that is, when two or more initial positions coincide) the conditioning becomes singular and some extra work is needed.

Although we do not need to deal explicitly with $\wcX^{(T)}$ started at points $x\in\partial\Omega_N$, our construction of the mBM requires some control of the process uniformly over such $x$, which effectively requires us to understand the process started at the boundary.
We were somewhat surprised to find that, in our setting of general diffusions, this does not appear to have been addressed in the literature---or at least, we were unable to find any relevant results.
So we will provide a proof of the result we need, which is that the one-point distributions of the process are well defined, and absolutely continuous with respect to the Lebesgue measure, at any fixed time.

Since the periodicity plays no role in this result, we state it in a more general setting.
We begin with the SDE
\begin{equation}\label{e.sde-2}
\d\X(t) = g(\X(t))\d t + \d B_t,
\end{equation}
where the function $g$ is $C^\infty$ and it and all of its derivatives are bounded.
Next we consider a system $\X=(\X_1,\dotsc,\X_N)$ of independent solutions of \eqref{e.sde-2} started at $x\in\Omega_N$ and, for $T>0$, let $\wcX^{(T)}$ denote a process which has the law of this system conditioned on non-intersection on $[0,T]$.

\begin{prop}\label{p.mbm-bdry}
  Let $(x^n)_{n\geq0}\subseteq\Omega_N$ be such that $x^n\longrightarrow\bar x\in\partial\Omega_N$ as $n\to\infty$.
  For fixed $t>0$, let $\mathbf{P}^{\tts t,T}_{\!x^n}$ denote the law of $\wcX^{(T)}(t)$ with the process started at $x^n$.
  Then the sequence $(\mathbf{P}^{\tts t,T}_{\!x^n})_{n\geq0}$ converges weakly to a probability measure $\mathbf{P}^{\tts t,T}_{\!\bar x}$, which is supported on $\Omega_N$, and is absolutely continuous with respect to the Lebesgue measure on $\rr^N$.
\end{prop}

\begin{rem}
  The result in the proposition can be extended to multipoint distributions straightforwardly.
  In fact, the result also yields convergence in distribution of the whole process restricted to $(0,T]$, uniformly over compact subsets of this interval, because once the limiting measure (supported on $\Omega_N$) at any fixed time $\delta>0$ is constructed, the distribution of the trajectory after time $\delta$ is defined by the standard conditioning mentioned above.
  Extending this to convergence in distribution over all of $[0,T]$ requires some additional work, which we will not pursue in this paper.
\end{rem}

Before turning to the proof, we need to state a classical PDE result.
Let $p_t(x,y)$ be the density (in $y$) of the solution of \eqref{e.sde-2} started at $x$, which exists because $g$ is $C^\infty$.
Then $p_t(x,y)$ solves the Kolmogorov forward (or Fokker-Plank) equation
\[\frac{\p}{\p t}p_t(x,y)=\frac12\frac{\partial^2}{\partial y^2}p_t(x,y)-g(y)\frac{\partial}{\partial y}p_t(x,y)\]
with initial condition $\lim_{t\searrow0}p_t(x,y)=\delta_x(y)$.
The results in \cite[Sec. 9.6]{friedman} and the fact that $g$ is $C^\infty$ imply that $p_t$ is $C^\infty$ in $x$ and $y$ and that for any $t>0$ and any $x\in\rr$, and for each $k\geq0$, there are constants $c,C>0$ so that
\begin{equation}\label{e.friedman-bd}
    \big|\tfrac{\p^k}{\p y^k}p_t(x,y)\big|\leq C e^{-c(x-y)^2}.
\end{equation}
Moreover, the constants $c$ and $C$ can be chosen uniformly over $x$ in compact sets (this is not stated explicitly in \cite{friedman}, but clearly follows from its arguments).

\begin{proof}[Proof of Proposition \ref{p.mbm-bdry}]
  By the Karlin-McGregor formula (see e.g. \cite{grabiner} for a statement that applies in our setting) we have that, for $x,y\in\Omega_N$, 
  \begin{equation}\label{e.KM}
    \det\!\big[p_t(x_i,y_j)\big]_{i,j=1}^N=\frac{\pp_{x}\big(\X(t)\in dy,\,\X(s)\in\Omega_N\;\forall\ts s\in[0,t]\big) }{dy }.
  \end{equation}
  Then, given any $f\!:\Omega_N\longrightarrow\rr$ which is continuous and bounded, and assuming $t<T$, we have
  \begin{equation}\label{e.xn-det-4}
    \ee_{x^n}\tsm\big(f(\wcX^{(T)}(t))\big)=\frac{\int_{\Omega_N}\!\int_{\Omega_N}f(y)\det\!\big[p_t(x^n_i,y_j)\big]_{i,j=1}^N\det\!\big[p_{T-t}(y_i,z_j)\big]_{i,j=1}^N\,dz\,dy}{\int_{\Omega_N}\!\det\!\big[p_T(x^n_i,y_j)\big]_{i,j=1}^N\,dy}
  \end{equation}
  by the Markov property.
  If $t\geq T$ then the numerator has to be modified suitably, but nothing changes in the argument that follows, so we will restrict to the case $t<T$.

  Suppose first for simplicity that we have $\bar x_1=\bar x_2$ but the remaining coordinates of $\bar x$ are different.
  In this case we can divide both the numerator and the denominator on the right hand side of \eqref{e.xn-det-4} by $x^n_2-x^n_1$ and absorbe that factor in the second row of each determinant involving $x^n$.
  Focusing on the numerator, if we subtract the first row divided by the same factor from the second row in that determinant, the resulting row has entries $\frac{p_t(x^n_2,y_j)-p_t(x^n_1,y_j)}{x^n_2-x^n_1}$.
  This converges pointwise to $p_t'(\bar x_1,y_j)$ as $n\to\infty$, while it is bounded above by $\sup_{\xi\in[-L,L]}|p_t'(\xi,y_j)|$ for some $L>0$.
  Thus, and thanks to \eqref{e.friedman-bd}, we may use the dominated convergence theorem to pass the limit as $n\to\infty$ inside the integral to get a limiting determinant with a modified second row.
  We can proceed in the same way for the denominator, and hence writing $p_{i,t}=p_t$ if $i\neq2$ and $p_{2,t}=p'_t$, we get
  \begin{equation}\label{e.fxn-lim}
    \ee_{x^n}\!\big(f(\wcX^{(T)}(t))\big)\xrightarrow[n\to\infty]{}\frac{\int_{\Omega_N}\int_{\Omega_N}f(y)\det\!\big[p_{i,t}(\bar x_i,y_j)\big]_{i,j=1}^N\det\!\big[p_{T-t}(y_i,z_j)\big]_{i,j=1}^N\,dz\,dy}{\int_{\Omega_N}\!\det\!\big[p_{i,T}(\bar x_i,y_j)\big]_{i,j=1}^N\,dy},
  \end{equation}
  assuming of course that the denominator on the right hand side does not vanish.
  The result which we are trying to prove clearly follows from \eqref{e.fxn-lim}, if the denominator is not zero. Thus, we concentrate in showing that that condition holds.

  Given any $h\in(0,T)$ we have
  \[\int_{\Omega_N}\!\det\!\big[p_{i,T}(\bar x_i,y_j)\big]_{i,j=1}^N\,dy
  =\int_{\Omega_N}\int_{\Omega_N}\det\!\big[p_{i,h}(\bar x_i,\xi_j)\big]_{i,j=1}^N\det\!\big[p_{T-h}(\xi_i,y_j)\big]_{i,j=1}^N\,d\xi\,dy,\]
  again by the Markov property and \eqref{e.KM} (this follows alternatively also from the generalized Cauchy-Binet, or Andr\'eief, identity \cite{andreief}).
  The second determinant on the right hand side is positive for all $\xi,y\in\Omega_N$ by \eqref{e.KM}, so the positivity of the left hand side will follow if we prove that, for some $h\in(0,T)$, $\det\!\big[p_{i,h}(\bar x_i,\xi_j)\big]_{i,j=1}^N$ is non-negative for all $\xi\in\Omega_N$ and is positive for some $\xi\in\Omega_N$.

  The first condition is clearly true for any $h>0$, because the determinant is the limit as $n\to\infty$ of $\frac1{x^n_2-x^n_1}\det\!\big[p_{h}(x^n_i,\xi_j)\big]_{i,j=1}^N$, which is itself non-negative by \eqref{e.KM}.
  For the second one, let $f_1,\dotsc,f_N$ be bounded $C^\infty$ functions with compact support.
  Then, writing $\tilde f_i=f_i$ for $i\neq2$ and $\tilde f_2=f_2'$, we have
  \begin{equation}\label{e.lim-fullint-det}
    \lim_{h\to0}\int_{\rr^N}\det\!\big[p_{i,h}(\bar x_i,\xi_j)\big]_{i,j=1}^N\prod_{j=1}^Nf_j(\xi_j)\,d\xi
    =\det\!\big[\tilde f_{i}(\bar x_j)\big]_{i,j=1}^N,
  \end{equation}
  as can be seen e.g. by expanding the determinant on the left hand side by definition and then computing the limit of each of the $N$ integrals.
  Now we can choose the $f_i$'s so that each of their supports contains $\bar x_i$ but none of the other $\bar x_j$'s, so that
  \[\det\!\big[\tilde f_{i}(\bar x_j)\big]_{i,j=1}^N=W(f_1,f_2)(\bar x_1)\prod_{j=3}^Nf_j(\bar x_j),\]
  where $W(f_1,f_2)$ denotes the Wronksian, i.e. $W(f_1,f_2)(x)=\det\Big[\begin{smallmatrix}f_1(x) & f_2(x)\\ f_1'(x) & f_2'(x)\end{smallmatrix}\Big]$.
  If we choose $f_1$ and $f_2$ so that the Wronskian is not zero, then the right hand side of \eqref{e.lim-fullint-det} is non-zero, which means that if $h$ is small enough then there has to be some $\xi\in\rr^N$ so that $\det\!\big[p_{i,h}(\bar x_i,\xi_j)\big]_{i,j=1}^N\neq0$.
  Note that the entries of this $\xi$ are necessarily different, and then by antisymmetry of the determinant we conclude that we may choose $\xi$ to be in $\Omega_N$.
  But for such $\xi$ we know already that the determinant is non-negative, so we conclude that, as desired, it is strictly positive for that choice of $\xi$.

  For a more general $\bar x\in\p\Omega_N$ we can proceed in a similar way.
  If all points in a cluster $x^n_{i},\dotsc,x^n_{i+k-1}$ are coalescing as $n\to\infty$, we divide the determinants in \eqref{e.xn-det-4} by $\prod_{i\leq j<j'\leq i+k-1}(x^n_{j'}-x^n_j)$ and perform successive row operations in a way similar to what we did above to produce a sequence of rows with entries of the form $p_T(\bar x_i,y_j),p'_T(\bar x_i,y_j),\dotsc,p^{(k-1)}_{T}(\bar x_i,y_j)$ in the limit (see e.g. \cite{tracyWidom-noninter} for a similar argument).
  The denominator in the analog of the right hand side of \eqref{e.fxn-lim} can be shown to not vanish in the same way as above, by choosing the $f_i$'s appropriately (in particular, now the Wronskian $W(f_i,\dots,f_{i+k-1})(\bar x_i)$ needs to be non-zero).
  On the other hand, if there is more than one cluster of points coalescing as $n\to\infty$, we can simply repeat the above argument for each of the clusters to get determinants where multiple groups of rows are replaced by successive derivatives, and argue similarly; we omit the details.
\end{proof}

As we mentioned, our construction of the mBM will require us to have some uniform control on $\wcX^{(T)}$ over all initial data $x\in\Omega_N$.
What we actually need is to show that the distribution of the process at time $T$ is concentrated away from the boundary of $\Omega_N$, uniformly in the initial condition.
We state and prove the precise result next.
Although it is intuitively easy to understand it, in view of Proposition \ref{p.mbm-bdry}, its proof requires some work.
Thus, since it simplifies the argument a bit, in this part we go back to the periodic setting (see also Remark \ref{r.nonper}).

\begin{prop}\label{p.uniform_separation}
    Let $\wcX^{(T)}$ denote a system of $N$ Brownian motions in a periodic potential $v$ (i.e., $N$ solutions of \eqref{e.SDE}) conditioned on non-intersection on $[0,T]$.
    Then for every $\delta\in(0,1)$ there is a $\gamma>0$ such that for any initial condition $x\in\Omega_N$,
    \begin{align*}
    \pp_x\Big(\tsm\min_{i}\big(\wcX^{(T)}_i(T)-\wcX^{(T)}_{i-1}(T)\big)>\gamma\Big) \geq 1-\delta.
    \end{align*}
\end{prop}

\begin{proof}
Assume that the claim does not hold.
Then there is a $\delta\in(0,1)$, a sequence $\gamma_n\searrow0$ and a sequence of initial conditions $x^n\in\Omega_N$, so that
\begin{equation}\label{e.contraclaim}
    \pp_{x^n}\!\Big(\wcX^{(T)}(T)\in\Omega^{\gamma_n}_N\Big)<1-\delta
\end{equation}
for each $n>0$, where
\begin{equation}\label{e.Omega-gamma}
    \Omega^{\gamma}_N=\big\{x\in\Omega_N\!:\min_{i=1,\dotsc,N-1}(x_{i+1}-x_i)>\gamma\big\}
\end{equation}
(note that $\Omega^0_N=\Omega_N$).

Suppose first that there is a compact set $K\subseteq\rr^N$ so that $x^n\in\Omega_N$ for each $N$.
Then $x^n$ has a subsequence, which we still denote by $x^n$, converging to some $\bar x$ in the closure of $\Omega_N$.
Now fix some $\gamma>0$ and observe that, for large $n$,
\begin{equation}\label{e.xndet}
\pp_{x^n}\!\Big(\wcX^{(T)}(T)\in\Omega^{\gamma_n}_N\Big)\geq\pp_{x^n}\!\Big(\wcX^{(T)}(T)\in\Omega^{\gamma}_N\Big)=\frac{\int_{\Omega^{\gamma}_N}\tsm\det\!\big[p_T(x^n_i,y_j)\big]_{i,j=1}^N\ts dy}{\int_{\Omega_N}\tsm\det\!\big[p_T(x^n_i,y_j)\big]_{i,j=1}^N\ts dy},
\end{equation}
where the equality comes from the Karlin-McGregor formula \eqref{e.KM}.
If $\bar x\in\Omega_N$ then we have $\int_{\Omega_N}\tsm\det\!\big[p_T(\bar x_i,y_j)\big]_{i,j=1}^N\ts dy=\pp_{\bar x}(\X(T)\in\Omega_N)>0$ (where $\X$ is a vector of $N$ independent solutions of \eqref{e.SDE}), and thus the bound \eqref{e.friedman-bd} and the dominated convergence theorem imply that we may take $n\to\infty$ and pass the limit inside the integrals in both the denominator and the numerator on the right hand side of \eqref{e.xndet} to get 
\begin{equation}
    \liminf_{n\to\infty}\pp_{x^n}\!\Big(\wcX^{(T)}(T)\in\Omega^{\gamma_n}_N\Big)
    \geq\frac{\int_{\Omega^{\gamma}_N}\tsm\det\!\big[p_T(\bar x_i,y_j)\big]_{i,j=1}^N\ts dy}{\int_{\Omega_N}\tsm\det\!\big[p_T(\bar x_i,y_j)\big]_{i,j=1}^N\ts dy}.\label{e.xndet2}
\end{equation}
But since $\det\!\big[p_T(\bar x_i,y_j)\big]_{i,j=1}^N$ vanishes at $y\in\p\Omega_N$, we can make this ratio be as close to $1$ as we want by choosing a small enough $\gamma$, which contradicts \eqref{e.contraclaim}.
Similarly, if $\bar x\in\p\Omega_N$ then Proposition \ref{p.mbm-bdry} allows us similarly to take $n\to\infty$ to get, in the notation of that result, 
\[\liminf_{n\to\infty}\pp_{x^n}\!\Big(\wcX^{(T)}(T)\in\Omega^{\gamma_n}_N\Big)\geq\mathbf{P}^{\tts T,T}_{\!\bar x}\tsm(\Omega^\gamma_N),\]
which can be made arbitrarily close to $1$ for small enough $\gamma$ in the same way.

Now we turn to the general case when the sequence $x^n$ is not contained in a compact set.
If the gaps between successive coordinates $x^n_i$ are bounded (i.e., if there is an $M>0$ so that $x^n_{i+1}-x^n_i\leq M$ for each $i=1,\dotsc,N-1$ and all $n\geq1$) then, since the drift term $v'$ in \eqref{e.SDE} is periodic, we can simply translate the whole system by $-x^n_1$ to recover a situation where the initial conditions stay in a compact set, for which the previous argument applies.

Now suppose that there is an index $k$ so that $x^n_{k+1}-x^n_{k}\longrightarrow\infty$ as $n\to\infty$.
For notational simplicity we will assume that this is the only gap which is going to infinity, i.e. that the clusters of particles with indices $J_-=(1,\dotsc,k)$ and $J_+=(k+1,\dotsc,N)$ each have gaps which stay bounded as $n\to\infty$; as will be clear, the argument for the general case can be adapted easily from the one we will present.

The basic intuition behind the proof is that, as they get further away from each other (for large $n$), the two clusters should behave roughly independently, and since the gaps between particles in each cluster are bounded, the argument for the compact case should apply.
However, due to the conditioning, we need to implement this idea with some care.

Recall that $\X=(\X_1,\dotsc,\X_N)$ denotes a system of $N$ independent solutions of the SDE \eqref{e.SDE}. We use the definition of $\wcX^{(T)}$ to see that for any $\gamma>0$
\[\pp_{x^n}\!\big(\wcX^{(T)}(T)\in\Omega^{\gamma}_N\big)=\frac{\pp_{x^n}\big(\X(t)\in\Omega_N\;\forall\,t\leq T,\;\X(T)\in\Omega^{\gamma}_N\big)}{\pp_{x^n}\big(\X(t)\in\Omega_N\;\forall t\leq T\big)}.\]
For the numerator we first note that, as before, and by the periodicity of $v'$, we may translate the system so that $x^n_{k}\in[-1/2,1/2)$ for all $n$.
Assuming that condition, there is now a sequence of integers $\ell_n\to\infty$ so that $x^n_{k+1}-\ell_n\in[1/2,3/2)$.
Then, writing $\Omega^\gamma_-=\Omega^{\gamma}_{k}$ and $\Omega^\gamma_+=\Omega^{\gamma}_{N-k}$
\begin{align}
  &\pp_{x^n}\big(\X(t)\in\Omega_N\;\forall\,t\leq T,\;\X(T)\in\Omega^{\gamma}_N\big)
  =\int_{\Omega^{\gamma}_N}\!\det\!\big[p_T(x^n_i,y_j)\big]_{i,j=1}^N\,dy\\
  &\hspace{0.8in}=\int_{\Omega^{\gamma}_-\times\Omega^{\gamma}_+}\!\det\!\big[p_T(x^n_i,y_j)\big]_{i,j=1}^N\ts\uno{y_{k+1}-y_k>\gamma}\,dy\\
  &\hspace{0.8in}=\int_{\Omega^{\gamma}_-\times\Omega^{\gamma}_+}\!\det\!\big[p_T(x^n_i,y_j+\ell_n\uno{j>k})\big]_{i,j=1}^N\ts\uno{y_{k+1}-y_k>\gamma-\ell_n}\,dy\\
  &\hspace{0.8in}=\int_{\Omega^{\gamma}_-\times\Omega^{\gamma}_+}\!\det\!\big[p_T(x^n_i-\ell_n\uno{i>k},y_j+\ell_n(\uno{j>k}-\uno{i>k})\big]_{i,j=1}^N\ts\uno{y_{k+1}-y_k>\gamma-\ell_n}\,dy.
\end{align}
where in the first equality we used the Karlin-McGregor formula \eqref{e.KM} again, 
in the third one we shifted the variables $y_{k+1},\dotsc,y_N$ by $\ell_n$, and in the last one we used the fact that $v'$ has period $1$ so the diffusion is invariant in distribution under integer translations.
To estimate the denominator, we write $y_{J_\pm}=(y_i)_{i\in J^\pm}$ for a vector $y\in\rr^N$, and define the events
\[G^\gamma_{J_\pm}=\big\{\X_{J_\pm}(t)\in\Omega^0_\pm\;\forall\,t\leq T,\;\X_{J_\pm}(T)\in\Omega^\gamma_\pm\big\}.\]
Then 
\[\pp_{x^n}\big(\X(t)\in\Omega_N\;\forall\,t\leq T,\;\X(T)\in\Omega_N\big)\leq\pp_{x^n_{J_-}}\!\tsm\big(G^0_{J_-}\big)\pp_{x^n_{J_+}}\!\tsm\big(G^0_{J_+}\big)=\pp_{x^n_{J_-}}\!\tsm\big(G^0_{J_-}\big)\pp_{x^n_{J_+}\!\!-\ell_n}\!\tsm\big(G^0_{J_+}\big),\]
by independence of the coordinates and the fact that $\Omega_N\subseteq\Omega_{-}\times\Omega_{+}$, together with invariance under integer translations again.
Writing $\tilde{x}^n_i=x^n_i-\ell_n\uno{i>k}$, we deduce from the identity for the numerator and the estimate for the numerator that
\begin{equation}\label{e.XTTGG}
  \pp_{x^n}\tsm\big(\wcX^{(T)}(T)\in\Omega^{\gamma}_N\big)\geq\frac{\int_{\Omega^{\gamma}_{-}\times\Omega^{\gamma}_{+}}\det\!\big[p_T(\tilde x^n_i,y_j+\ell_n(\uno{j>k}-\uno{i>k}))\big]_{i,j=1}^N\ts\uno{y_{k+1}-y_k>\gamma-\ell_n}\,dy}{\pp_{\tilde{x}^n_{J_-}}\!\tsm\big(G^0_{J_-}\big)\pp_{\tilde{x}^n_{J_+}}\!\tsm\big(G^0_{J_+}\big)}.
\end{equation}

Recall now that we are assuming that the gaps between particles within $\tilde{x}^n_{J_-}$ and within $\tilde{x}^n_{J_+}$ remain bounded as $n\to\infty$.
Then, since the rightmost point of $\tilde x^n_{J_-}$ and the leftmost point of $\tilde x^n_{J_+}$ are, respectively, in $[-1/2,1/2)$ and $[1/2,3/2)$, both $\tilde x^n_{J_-}$ and $\tilde{x}^n_{J_+}$ are contained inside compact sets, and thus we can extract a common subsequence out of them, which we still denote by $\tilde{x}^n_{J_\pm}$, and which converges to some limits $\bar x_{J_-}$ and $\bar x_{J_+}$.
If $\bar x_{J_\pm}\in\Omega_\pm$, then taking $n\to\infty$ in the above estimate and proceeding as above (using \eqref{e.friedman-bd} in particular, which implies that if $x_n\longrightarrow x$ and $y_n\longrightarrow\pm\infty$ then $p_T(x_n,y_n)\longrightarrow0$) we get that, for any fixed $\gamma>0$,
\begin{align}
  \liminf_{n\to\infty}\pp_{x^n}\tsm\big(\wcX^{(T)}(T)\in\Omega^{\gamma_n}_N\big)
  &\geq\liminf_{n\to\infty}\pp_{x^n}\tsm\big(\wcX^{(T)}(T)\in\Omega^{\gamma}_N\big)\\
  &=\frac{\int_{\Omega^\gamma_-\times\Omega^\gamma_+}\det\!\big[p_T(\bar x_i,y_j)\ts\uno{(i,j)\in J_-\uptext{ or }(i,j)\in J_+}\big]_{i,j=1}^N\,dy}{\pp_{\bar x_{J_-}}\!\tsm\big(G^0_{J_-}\big)\pp_{\bar{x}_{J_+}}\!\tsm\big(G^0_{J_+}\big)}\label{e.no-compacto-pero-sin-choque}\\
  &=\frac{\int_{\Omega^\gamma_-\times\Omega^\gamma_+}\det\!\big[p_T(\bar x_i,y_j)\big]_{i,j\in J_-}\det\!\big[p_T(\bar x_i,y_j))\big]_{i,j\in J_+}\,dy}{\pp_{\bar x_{J_-}}\!\tsm\big(G^0_{J_-}\big)\pp_{\bar{x}_{J_+}}\!\tsm\big(G^0_{J_+}\big)}\\
  &=\frac{\pp_{\bar x_{J_-}}\big(G^\gamma_{J_-}\big)\pp_{\bar x_{J_+}}\big(G^\gamma_{J_+}\big)}{\pp_{\bar x_{J_-}}\!\tsm\big(G^0_{J_-}\big)\pp_{\bar{x}_{J_+}}\!\tsm\big(G^0_{J_+}\big)}.
\end{align}
The right hand side can be made as close to $1$ as we want again by choosing $\gamma$ to be small, contradicting \eqref{e.contraclaim} as before.

If $\bar{x}_{J_-}\in\partial\Omega_-$ or $\bar{x}_{J_+}\in\partial\Omega_+$, we need to combine the above computation with the argument which we used in the proof of Proposition \ref{p.mbm-bdry}.
For simplicity, and to compare easily with that proof, we will assume that $\bar x_1=\bar x_2$ and $\bar x_k=\bar x_{k+1}$ while all remaining points are distinct; the extension to the general case can be done in the same way as in that proof.
We divide the numerator and the denominator on the right hand side of \eqref{e.XTTGG} by $(x^n_2-x^n_1)(x^n_{k+1}-x^n_k)$ and then take $n\to\infty$ as in the proof of that proposition to get
\begin{equation}
  \liminf_{n\to\infty}\pp_{x^n}\tsm\big(\wcX^{(T)}(T)\in\Omega^{\gamma_n}_N\big)
  \geq\frac{\int_{\Omega^\gamma_-\times\Omega^\gamma_+}\det\!\big[p_{i,T}(\bar x_i,y_j)\uno{(i,j)\in J_-\uptext{ or }(i,j)\in J_+}\big]_{i,j=1}^N\,dy}{\left(\int_{\Omega_-}\det\!\big[p_{i,T}(\bar x_i,y_j)\big]_{i,j=1}^k\,dy\right)\left(\int_{\Omega_+}\det\!\big[p_{i,T}(\bar x_i,y_j)\big]_{i,j=k+1}^N\,dy\right)}
\end{equation}
where, analogously to the notation in \eqref{e.fxn-lim}, $p_{i,T}(x,y)$ denotes either $p_T(x,y)$ for $i\notin\{2,k+2\}$ and $p'_T(x,y)$ for $i=2,k+1$.
The key points in this computation are: (i) Each factor on the denominator of \eqref{e.XTTGG} can be expressed as the integral of a determinant by the Karlin-McGregor formula and the two factors by which we are dividing, $x^n_2-x^n_1$ and $x^n_{k+1}-x^n_k$, can be absorbed into the second row of the corresponding determinants to get the desired limit; (ii) The factors can similarly be absorbed into rows $2$ and $k+2$ of the determinant in the numerator, leading again to the stated limit; (iii) We know already from the proof of Proposition \ref{p.mbm-bdry} that the two factors in the limiting denominator are positive; and (iv) The estimate \eqref{e.friedman-bd} holds for $p_T$ and its derivatives, so the argument which gave vanishing coefficients in \eqref{e.no-compacto-pero-sin-choque} away from the $(J_-,J_-)$ and $(J_+,J_+)$ blocks applies here too.
The determinant in the numerator factors as in the last computation, and then in the notation of Proposition \ref{p.mbm-bdry} we have obtains
\begin{equation}
  \liminf_{n\to\infty}\pp_{x^n}\tsm\big(\wcX^{(T)}(T)\in\Omega^{\gamma_n}_N\big)
  \geq\pp_{\bar x_{J_-}}^{T,T}(\Omega_-^\gamma)\pp_{\bar x_{J_+}}^{T,T}(\Omega_+^\gamma).
\end{equation}
This contradicts \eqref{e.contraclaim} once again because the right hand side can be made arbitrarily close to $1$ if $\gamma$ is small, finishing the proof.
\end{proof}

\begin{rem}\label{r.nonper}
    In the last proof we used the fact that the drift $v'$ is periodic in order to translate the diffusion (several times) without changing its law.
    If $v'$ was not periodic, one can still shift the initial data, as long as the drift $v'$ is shifted accordingly.
    In that case, if the supremum norms of the drift and its first $N$ derivatives are bounded, then in principle the argument could be pushed through if one had suitable control on $p_T$ and its derivatives which is uniform over choices of drifts satisfying the same supremum norm bounds.
    In the case of the estimate \eqref{e.friedman-bd}, this uniformity is not part of the statement of Theorem 6.7 of \cite{friedman}, but seems to follow from its proof (and is presumably known in the literature).
\end{rem}

\section{Construction of myopic non-intersecting random walks and Brownian motions}

\subsection{Construction of the mRW}\label{s.mRW}

In this subsection, we define the non-intersecting myopic random walk $\Y^{(L)}$, prove Theorem \ref{t.interpolation}, and then provide a direct construction of $\Y^{(L)}$ which will play a crucial role in the rest of the paper.
Although in this paper we only need to consider mRWs constructed out of totally asymmetric walks (i.e., Poisson processes), it will be instructive, and not any more complicated, to consider here the case of walks which can jump in both directions.

\subsubsection{Definition as a Markov process}\label{s.mrW_Markov_generator}

For $N\in\nn$ and $p\in[0,1]$, let $\Y= (\Y_1, \Y_2, ...,  \Y_N)$ be a system of independent continuous time random walks jumping to the right at rate $p$ and to the left at rate $1-p$.
Now for fixed $L>0$, introduce the function
\begin{align*}
h_L(y)=\pp_{y}(\Y(t)\in \Omega_N, \forall t\in [0,L]).
\end{align*}
Note that $h_L(y)>0$ for $y\in\Omega_N^\zz$.

We may define the mRW as follows:

\begin{defn}[Myopic non-intersecting random walks] 
Let $N\in \nn$, $L>0$, and $p\in [0,1]$.
The system of \emph{myopic of non-intersecting random walks with foresight $L$} (mRW) is the time-homogeneous Markov chain $\Y^{(L)}=(\Y^{(L)}_1,\Y^{(L)}_2,...,\Y^{(L)}_N)$ taking values in the discrete Weyl chamber $\Omega_N^\zz$, whose generator is given by
\begin{equation}\label{e.generator_mRW}
\cL^{(L)} F(y)=\sum_{i=1}^N\left(p\tts(F(y+e_i)-F(x))\frac{h_L(y+e_i)}{h_L(x)} + (1-p)(F(y-e_i)-F(x))\frac{h_L(y-e_i)}{h_L(x)}\right).
\end{equation}
\end{defn}

One may think about \eqref{e.generator_mRW} is as an analog of a Doob $h$-transform of the generator of the system of independent walks.
Note, however, that $h_L$ is not harmonic for this process.

Intuitively, the infinitesimal evolution of the mRW at time $t$ is that of the independent system $(\Y(t))_{t\geq0}$, but conditioned on non-intersection during the interval $[t,t+L]$.
Note that this time window for non-intersection is shifted dynamically as the process evolves.
The following observation provides one possible way to make sense of this:

\begin{rem}\label{r.checkY}
Another $\Omega_N$-valued process which will be important in this paper is the process $(\wcY^{(L)}(t))_{t\geq 0}$ defined to have the law of the homogeneous Markov chain $\Y$, but conditioned on belonging to $\Omega_N$ up to time $L$. 
This is a time-inhomogeneous Markov process, and one can show that its generator at time $0$ coincides with $\mathcal L^{(L)}F(x)$.
\end{rem}

Now we will prove that this system interpolates between ASEP and a system of non-intersecting random walks.
In the case $p=1$, this corresponds to Theorem \ref{t.interpolation}.

\begin{prop}
    The limits $\Y^{(0^+)}=\lim_{L\searrow0}\Y^{(L)}$ and $\Y^{(\infty)}=\lim_{L\nearrow\infty}\Y^{(L)}$ exist in distribution in the local Skorohod topology.
    Furthermore, $\Y^{(0^+)}$ has the law of ASEP with $N$ particles and jump rate $p$ to the right and $1-p$ to the left, while $\Y^{(\infty)}$ has the law of a system of random walks $\Y= (\Y_1, \Y_2, \dotsc, \Y_N)$ conditioned on never intersecting.
\end{prop}

\begin{proof}
We are dealing with non-explosive continuous time Markov chains on a countable state space, so it is enough to check the convergence of the generators.
More precisely, we need to show that if $F\!:\Omega_N^\zz\longrightarrow\rr$ is bounded, then $\cL^{(L)}F$ converges pointwise to $\bar\cL F$ in each of the two limits in $L$, with $\bar\cL$ the generator of the limiting process (e.g. by Lemma \ref{l.increments_enough_Skorohod} below).

For the first one simply note that, for $x\in\Omega_N^\zz$, 
$\frac{h_L(x\pm e_i)}{h_L(x)}\longrightarrow\uno{x\pm e_i\in\Omega_N^\zz}$ as $L\searrow0$, so
$\cL^{(L)}F(x)$ converges to
\[\sum_{i=1}^N\left(p\tts(F(x+e_i)-F(x))\uno{x+e_i\in\Omega_N^\zz} + (1-p)(F(x-e_i)-F(x))\uno{x-e_i\in\Omega_N^\zz}\right),\]
which is the generator of ASEP with $N$ particles.

For the second case we first invoke \cite[Thm. 1.1]{eichelsbacherKonig}, which gives $L^{\frac{1}{4}(N-1)N}h_L(x)\longrightarrow\Delta(x)$ as $L\to\infty$ (with $\Delta(x)$ the Vandermonde determinant).
From this we get
\[\frac{h_L(x\pm e_i)}{h_L(x)}\xrightarrow[L\to\infty]{}\frac{\Delta(x\pm e_i)}{\Delta(x)},\]
which gives convergence to the generator \eqref{e.generator_niRW}.
\end{proof}

\subsubsection{A direct construction of the mRW}\label{s.direct_mRW}

Next we present a construction of the mRW that will later allow us to couple it with the mBM.
The construction is based on a sort of acceptance-rejection algorithm.

The basic idea is the following.
Fix a foresight parameter $L>0$ and an initial condition $y\in\Omega_N^\zz$, and consider a process $\wcY^{(L)}$ which has the law of the system of independent walks $\Y$ started at $y$, conditioned on the event that $\Y_t\in\Omega_N^\zz$ for all $t\in [0,L]$ (as in Remark \ref{r.checkY}).
Now let $\tau>L$ be the first collision time between the walks (i.e., the exit time of $\wcY^{(L)}$ from $\Omega_N^\zz$) and note that for any $t\in[0,\tau-L)$, the trajectory of $\wcY^{(L)}$ starting at time $t$ remains inside $\Omega_N^\zz$ for the next $L$ units of time.
Then we keep that piece of the trajectory (i.e. $\wcY^{(L)}\big|_{[0,\tau-L]}$), discard the rest, and repeat the same construction starting where (and when) the previous piece of trajectory left off.

More precisely, we propose the following algorithm to construct an mRW $\Y^{(L)}$ with foresight $L$ and initial condition $y\in\Omega_N^\zz$ out of independent copies of $\wcY^{(L)}$:

\begin{algo}\label{a.A}
\leavevmode
\begin{enumerate}
    \item Let $t_0=0$ and set $\Y^{(L)}(t_0)=y$.
    \item For $n\geq0$, assume that we have defined $\Y^{(L)}$ up to time $t_n$. 
    Sample a copy $\wcY^{(L,n+1)}$ of $\wcY^{(L)}$ with initial condition $\Y^{(L)}(t_n)$, which is otherwise independent of the previous copies $\wcY^{(L,1)},\dotsc,\wcY^{(L,n)}$.
    Define $\tau_{n+1}$ to be first collision time of $\wcY^{(L,n+1)}$,
    \[\tau_{n+1}=\inf\{t>0: \wcY^{(L,n+1)}\notin\Omega_N^\zz\}.\]   
    Then let $t_{n+1} = t_{n}+\tau_{n+1}-L>t_n$ and set
    \[\Y^{(L)}\big|_{[t_n,t_{n+1}]}=\wcY^{(L,n+1)}\big|_{[0,t_{n+1}-t_n]}.\]
\end{enumerate}
\end{algo}

See Figure \ref{f.construction}, or \cite{simulation} for a simulation.

It is, of course, not obvious at all that the process constructed by this algorithm is  Markovian, let alone that it is time-homogeneous.

\begin{figure}
    \centering
    \includegraphics[width=0.6\linewidth]{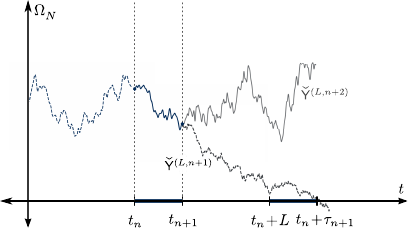}
    \caption{The construction of the mRW through Algorithm \ref{a.A}. 
    A single path (represented here as a continuous trajectory) is employed to represent the $N$-dimensional trajectories of the system, with the horizontal axis standing for the boundary of $\Omega_N$.
    Starting at time $t_n$, we sample a system $\wcY^{(L,n+1)}$ conditioned on non-intersection on the interval $[t_n,t_n+L]$, and then keep only the initial part of its trajectory between times $t_n$ and $t_{n+1}=t_n+\tau_{n+1}-L$ (whose length is represented by the two blue segments).
    At time $t_{n+1}$, the construction is restarted using an independent copy $\wcY^{(L,n+2)}$ of the system conditioned to not intersect for time $L$.
    The same sketch can be used to represent the construction of an mBM through Algorithm \ref{a.B}.}
    \label{f.construction}
\end{figure}

\begin{prop}\label{p.construccion_discreta}
The process $\Y^{(L)}$ constructed by Algorithm \ref{a.A} defines a process $\Y^{(L)}_t$ for all $t\geq0$, which has the law of an mRW.
\end{prop}

\begin{proof}
    We first need to check that $t_n\nearrow \infty$ almost surely.
    But this follows simply from the fact that $t_{n+1}-t_n\geq \xi_n$, where $\xi_n$ is the first jump after $L$ in the trajectory of $(\wcY^{(L,n+1)})_{t\geq0}$. Note that $(\xi_n)_{n\in \rr}$ are i.i.d. exponential random variables with parameter $N$.

    Now we prove that the process $\Y^{(L)}$ is a time-homogeneous Markov process (throughout this proof we  use $\Y^{(L)}$ to denote the process constructed through Algorithm \ref{a.A}).
    To do this we will show that the law of $\Y^{(L)}\big|_{[t,\infty)}$ given $\mathcal F_t\coloneqq\sigma(\Y^{(L)}(s)\!:s\leq t)\vee \sigma(t_n\mathbf 1_{t_n\leq t}\!:n\in \rr)$ coincides with the law of an independent copy of $\Y^{(L)}$ started at the configuration of the process at time $t$.
    This is enough, because the filtration $(\mathcal F_s)_{s\geq0}$ is larger than the natural filtration associated to $\Y^{(L)}$. 

    For fixed $t>0$, we can construct $\Y^{(L)}\big|_{[t,\infty]}$ conditionally on $\mathcal F_t$ as follows.
    Given $\mathcal F_t$, we can choose the value of $n$ so that $t_n\leq t<t_{n+1}$.
    Note that between times $t_n$ and $t_{n+1}$ we are tracing the trajectory of a process $\Z$ which has the law of $\wcY^{(L)}$ starting at $\Y^{(L)}(t_n)$.
    When we condition on $\mathcal F_t$, the only information we get in addition to the restriction of $\Z$ to $[0,t-t_n]$ is the fact that $t<t_{n+1}$ or, what is the same, $\tau_{n+1}>L+t-t_n$, which says that $\Z$ has not exited the Weyl chamber up to time $t-t_n+L$.
    As a consequence, under this conditioning the law of $\Z\big|_{[t-t_n,t-t_n+L]}$ is that of $\wcY^{(L)}\big|_{[0,L]}$ for an independent copy of $\wcY^{(L)}$ starting at $\Y^{(L)}(t)$.
    By construction of Algorithm \ref{a.A}, this implies that:
    \begin{enumerate}[label=(\roman*),itemsep=2pt]
        \item The law of $\Y^{(L)}\big|_{[t,t_{n+1}]}$ (or, what is the same, that of $\Z\big|_{[t-t_n,t_{n+1}-t_n]}$) is that of $\wcY^{(L)}\big|_{[0,t_{n+1}-t_n]}$ for an independent copy of $\wcY^{(L)}$ starting at $\Y^{(Y)}(t)$, and
        \item The law of $\Y^{(L)}\big|_{[t_{n+1},\infty)}$ is that of the whole trajectory of a copy of $\Y^{(L)}$ starting at $\Y^{(L)}(t_{n+1})$.
    \end{enumerate}
    This proves that $\Y^{(L)}$ is a time-homogeneous Markov process.

    To finish the proof it is enough to compute the generator of $\Y^{(L)}$.
       For bounded $F\!:\Omega_N^\zz\longrightarrow\rr$ we have 
    \begin{multline}
    \label{e.gen}
        \ee_y\!\left[F(\Y^{(L)}(h))-F(y)\right ]
        =  \ee_y\!\left[\big(F(\Y^{(L)}(h))-F(y)\big) \mathbf 1_{\Y^{(L)}(h)\neq \Y^{(L)}(0),\,t_{1}\leq h } \right ]\\ 
        + \ee_y\!\left[\big(F(\Y^{(L)}(h))-F(y)\big) \mathbf 1_{\Y^{(L)}(h)\neq\Y^{(L)}(0),\,t_{1}>h}\right].
    \end{multline}
    The first term on the right hand side is negligible because $\big\{\Y^{(L)}(h)\neq\Y^{(L)}(0),\,t_{1}\leq h\big\}$ is contained in the event 
    \begin{multline*}
    E_h=\big\{\wcY^{(L,1)} \text{ jumps both in $[0,h]$ and in $[L,L+h]$} \big\}\\
     \cup \big\{\wcY^{(L,1)} \text{ jumps in $[L,L+h]$ and $\wcY^{(L,2)}$ jumps in $[0,h]$ or $[L,L+h]$}\big\}, 
    \end{multline*}
    which  has probability of order $h^2$. Thus,
    \begin{equation}
    \frac{1}{h}\left|\ee_y\!\left[\big(F(\Y^{(L)}(h))-F(\Y^{(L)}(0))\big) \mathbf 1_{\Y^{(L)}(h)\neq \Y^{(L)}(0), t_{1}\leq h }\right]\right|
    \leq \frac{2\|F\|_\infty}{h }\ts\pp\big(E_{h}\big)\xrightarrow[h\to0]{}0.    
    \end{equation}
    
    For the other term, we have 
    \begin{align*}
        \frac{1}{h}\ee\!\left[\big(F(\Y^{(L)}(h))-F(y)\big) \mathbf 1_{\Y^{(L)}({h})\neq \Y^{(L)}(0), t_{1}>h} \right ]&=\frac{1}{h}\ee\!\left[\big(F(\wcY^{(L,1)}(h))-F(y)\big) \mathbf 1_{t_{1}>h}\right ]\\
        &=\frac{1}{h}\ee\!\left[F(\wcY^{(L,1)}(h))-F(y)\right ]+o(1),
    \end{align*}
    where the $o(1)$ estimate comes from basically the same argument as for the previous term. Therefore, as $h\to0$ the last expression converges to
    \begin{equation}
      \sum_{i=1}^N\Big(p\big(F(y+e_i)-F(y\big)\tfrac{h_{L}(y+e_i)}{h_{L}(y)}
      +(1-p)\big(F(y-e_i)-F(y\big) \tfrac{h_{L}(y-e_i)}{h_{L}(\mathsf{z}_0)}\Big)=\cL^{(L)}F(y).
    \end{equation}
      
    Putting the two limits together,  dividing \eqref{e.gen} by $h$, and taking $h\to0$, we get
    \[\frac{\d}{\d t}\ee_y\tsm\big[F(\Y^{(L)})(t)\big]\Big|_{t=0}=\cL^{(L)}F(y),\]
    so the process $\Y^{(L)}$ constructed by the algorithm has the desired generator.
    This concludes the proof.
    \end{proof}

\subsection{Construction of the mBM}\label{s.mBM}

Now, we turn to the construction of the system of myopic non-intersecting Brownian motions in a periodic potential.
In all that follows, we assume that the potential $v$ satisfies the assumptions prescribed in Section \ref{s.mbm-intro}.

One approach to defining the system is to prescribe its generator, as described in Section \ref{s.myopic-intro} and analogously to what we did for the mRW in Section \ref{s.mrW_Markov_generator}.
This is certainly possible, but slightly more challenging technically now because we are working with diffusions.
Since we will have no use for this approach, we do not include the details.

The alternative is to construct the mBM directly from a system $\wcX^{(T)}$ of copies of the  diffusion \eqref{e.SDE} conditioned on not intersecting in a fixed interval $[0,T]$.
We will provide two different such constructions.
The first one, which will serve as our definition of the mBM, proceeds by concatening copies of $\wcX^{(T)}$ run for time $\ep$ and then sending $\ep\to0$.
The second one is a version of the acceptance-rejection algorithm (Algorithm \ref{a.A}) which we introduced for the mRW.
The advantage of this last approach is that it will allow us to easily couple mBMs and mRWs, which is the main tool we will use to prove Theorem \ref{t.main}.

We proceed now with the two constructions.

\subsubsection{Glueing}\label{s.glueing}

Suppose (slightly abusing notation) that $\X=(\X_1,\dotsc,\X_N)$ is a system of $N$ independent solutions of the SDE \eqref{e.SDE}, and assume that the initial condition $\X(0)=(\X_1(0),\dotsc,\X_N(0))$ lives in the Weyl chamber
\[\Omega_N=\big\{x\in\rr^N\!:x_1<\dotsm<x_N\big\}.\]
For a fixed foresight $T>0$, define $\wcX^{(T)}$ to be the process taking values in $\rr^N$ which has the law of $\X$ conditioned on the event that the $\X(t)\in \Omega_N$ for any $t\in [0,T]$. Note that $\wcX^{(T)}$ is a time inhomogeneous Markov process and that, straightforwardly, at any time $0\leq t\leq T$ we have that almost surely $\wcX(t)\in\Omega_N$.

For $\epsilon>0$, we define  $ \X^{(T,\eps)}$, the $\epsilon$ approximation of the mBM process, as follows:
\begin{itemize}
    \item  $\X^{(T,\eps)}\big|_{[0,\eps]}$ is equal in law to $\wcX^{(T)} \big|_{[0,\eps]}$.
    \item For any $n\geq1$, and conditionally on $(\X_\eps^{(T)}(n\eps))= y$, the law of $\X_\eps(\cdot+n_\eps)\big|_{[0,\eps]}$ is equal to the law of $\wcX^{(T)}\big|_{[0,\eps]}$ with initial condition $y$, and is (conditionally) independent of $\X^{(T,\eps)}\big|_{[0,n\eps)}$.
\end{itemize}
The following result states that this law has a distributional limit as $\eps \to 0$.

\begin{prop} \label{p.existence_mBM}
For any foresight $T>0$ we have that $\X^{(T,\eps)}$ converges in law to a continuous, time-homogeneous Markov process $\X^{(T)}$ in the topology of uniform convergence in compacts.
\end{prop}

\begin{defn}[Myopic non-intersecting Brownian motions in a potential] \label{d.mBM}
Let $N\in \nn$ and $T>0$.
The \emph{myopic system of non-intersecting Brownian motions in a potential $v$ with foresight $T$} (mBM) with initial condition $x\in\Omega_N$ is the $\Omega_N$-valued Markov process distributed according to the limiting law provided in Proposition \ref{p.existence_mBM}.
\end{defn}

The proof of Proposition \ref{p.existence_mBM} is based on the second construction of the mBM and a coupling argument.
So we turn to that construction before proving the result.

\subsubsection{Direct construction of the mBM}\label{sec:direc-mBM}

The construction which we give now is analogous to the one given for the mRW in Section \ref{s.direct_mRW}.
In fact, the algorithm to construct it is the same as for the mRW, but for clarity we include it explicitly:

\newpage

\begin{algo}\label{a.B} 
    \leavevmode
    \begin{enumerate}
    \item Let $t_0=0$ and set $\X^{(T)}(t_0)=x$.
    \item For $n\geq0$, assume that we have defined $\X^{(T)}$ up to time $t_n$. 
    Sample a copy $\wcX^{(T,n+1)}$ of $\wcX^{(T)}$ with initial condition $\X^{(T)}(t_n)$, which is otherwise independent of the previous copies $\wcX^{(T,1)},\dotsc,\wcX^{(T,n)}$.
    Define $\tau_{n+1}$ to be first collision time of $\wcX^{(T,n+1)}$,
    \[\tau_{n+1}=\inf\{t>0: \wcX^{(T,n+1)}(t)\notin\Omega_N\}.\]   
    Then let $t_{n+1} = t_n+\tau_{n+1}-T>t_n$ and set
    \[\X^{(T)}\big|_{[t_n,t_{n+1}]}=\wcX^{(T,n+1)}\big|_{[0,t_{n+1}-t_n]}.\]
    \end{enumerate}
\end{algo}

See again Figure \ref{f.construction}.

\begin{prop}\label{p.construccion_continua}
    The process $\X^{(T)}$ defined through Algorithm \ref{a.B} is almost surely well defined for all times $t\geq0$, and $(\X^{(T)}(t))_{t\geq0}$ is a time-homogeneous Markov process.
\end{prop}

\begin{proof}
We start by proving that $t_n\nearrow \infty$.
Thanks to Proposition \ref{p.uniform_separation}, together with the almost sure continuity of $\X$ and the fact that $v'$ is bounded, we know that there is an $\eta>0$ so that
\begin{equation}\label{e.t1_bd}
\pp_x(t_1\geq \eta)\geq 1/2
\end{equation}
for any $x\in\Omega_N$.
Repeating the argument at each time $t_n$ we similarly have that, for each $n\geq1$, $\pp_x(t_{n+1}-t_n\geq \eta)\geq 1/2$ for any $x\in\Omega_N$.
Since the collection $(t_{n+1}-t_n)_{n\geq1}$ is independent, this implies that $t_n\nearrow \infty$ almost surely.

To prove that $\X^{(T)}$ is a time-homogeneous Markov process, one can repeat word for word the argument which we used for $\Y^{(L)}$ in the proof of Proposition \ref{p.construccion_discreta} (after replacing $L$ with $T$ and $\Y$'s with $\X$'s).
\end{proof}

Of course, the whole point is that this algorithm constructs a process which is equal in law to the process which was introduced in Definition \ref{d.mBM}.
We state this as a proposition, which will be proved together with Proposition \ref{p.existence_mBM} in the next subsection:

\begin{prop}
    \label{p.equal_mBMs}
    The process constructed through Algorithm \ref{a.B} is an mBM, i.e. it has the law of the process introduced in Proposition \ref{p.existence_mBM} and Definition \ref{d.mBM}.
\end{prop}

\subsubsection{Proof of Propositions \ref{p.existence_mBM} and \ref{p.equal_mBMs}}

In this subsection, we prove Propositions \ref{p.existence_mBM} and \ref{p.equal_mBMs} in one go.
The first step is to provide a construction of $\X^{(T,\eps)}$ which is analogous to Algorithm \ref{a.B}:

\begin{algo}\label{a.C}
    \leavevmode
    \begin{enumerate}
    \item Let $t_0=0$ and set $\X^{(T,\ep)}(t_0)=x$.
    \item For $n\geq0$, assume that we have defined $\X^{(T,\ep)}$ up to time $t_n$. 
    Sample a copy $\wcX^{(T,n+1)}$ of $\wcX^{(T)}$ with initial condition $\X^{(T,\ep)}(t_n)$, which is otherwise independent of the previous copies $\wcX^{(T,1)},\dotsc,\wcX^{(T,n)}$.
    Define $\tau_{n+1}$ to be first collision time of $\wcX^{(T,n+1)}$,
    \[\tau_{n+1}=\inf\{t>0: \wcX^{(T,n+1)}\notin\Omega_N\}.\]   
    Then let $t_{n+1}=\eps \lceil \eps^{-1}(\tau_{n+1}-T+t_n)\rceil>t_n$ and set
    \[\X^{(T,\eps)}\big|_{[t_n,t_{n+1}]}= \wcX^{(T,n+1)}\big|_{[0,t_{n+1}-t_n]}.\]
\end{enumerate}
\end{algo}

Note that in this case all of the $t_n$'s are multiples of $\ep$. Furthermore, it construct $\X^{(T,\eps)}$.
\begin{lem}
    Algorithm \ref{a.C} constructs $\X^{(T,\ep)}$.
\end{lem}

\begin{proof}
    The proof is similar to that of Proposition \ref{p.construccion_continua}.
    In this case, since $\tau_{n+1}>T$ almost surely for all $n$, we have $t_{n+1}-t_n\geq\ep$ for all $n\geq0$, so $t_n\nearrow\infty$ almost surely.

    Let $\overline{\X}^{(T,\eps)}$ denote the process constructed by Algorithm \ref{a.C}.
    As before we have $t_1\geq\ep$, so by construction we get $\overline{\X}^{(T,\eps)}\big|_{[0,\ep]}=\wcX^{(T,1)}\big|_{[0,\ep]}$, which has the intended distribution for the process on the interval $[0,\ep]$.
    Now let $m\geq1$ and suppose that we already know that $\overline{\X}^{(T,\eps)}$ has the desired distribution on $[0,m\ep]$.
    Let $\mathcal{F}_{k\ep}=\sigma(\overline{\X}^{T,\eps}(s)\!:s\leq m\ep)\vee\sigma(t_n\uno{t_n\leq m\ep}\!:n\in\rr)$.
    We claim that the law of $\overline{\X}^{(T,\eps)}\big|_{[m\ep,(m+1)\ep]}$ given $\mathcal{F}_{m\ep}$ is the same as that of $\overline{\X}^{(T,\eps)}\big|_{[0,\ep]}$ for an independent copy of the process, started at $\overline{\X}^{(T,\eps)}(m\ep)$.
    As in the previous proof, this is enough because $\mathcal{F}_{m\ep}$ is larger than the natural filtration associated to $\overline{\X}^{(T,\eps)}$.
    And, similarly to that proof, the only information we get from conditioning on $\mathcal{F}_{m\ep}$ in addition to the trajectory of $\overline{\X}^{(T,\eps)}$ up to time $m\ep$ is the value of $n$ and of $t_n$ so that $t_n\leq m\ep<t_{n+1}$.
    If it happens that $t_n=m\ep$ then by construction of Algorithm \ref{a.C}, and for the same reason as for the argument for the interval $[0,\ep]$, the conditional distribution of the process on $[m\ep,(m+1)]$ is the desired one.
    If instead we have $t_n<m\ep<t_{n+1}$ then, as in the previous proof, we know that $\wcX^{(T,n+1)}$ has not exited the Weyl chamber up to time $m\ep-t_n+T$, and as before we deduce that $\overline{\X}^{(T,\eps)}$ has the specified distribution on $[m\ep,(m+1)\ep]$.
\end{proof}

We now have all the tools to prove the two propositions.

\begin{proof}[Proof of Propositions \ref{p.existence_mBM} and \ref{p.equal_mBMs}]
The proof proceeds by coupling. 
It is based on the following claim:

\begin{claim}\label{c.epsiloncoupling}
    Fix $u >0$, $\eta>0$ and $\gamma>0$. Then, there is a $\delta>0$ such that for all $x_0,y_0\in \Omega_N^{\gamma}$ (defined in \eqref{e.Omega-gamma}) with $|x_0-y_0|\leq \delta$, there exists a coupling between copies $\wcX^{(T)}_{x_0}$ and $\wcX^{(T)}_{y_0}$ of $\wcX^{(T)}$, starting at $x_0$ and $y_0$ respectively, such that
    \begin{align*}
    \pp\!\left (\wcX_{x_0}^{(T)}\big|_{[u,\infty)}=\wcX^{(T)}_{y_0}\big|_{[u,\infty)}\right ) \geq1-\eta.
    \end{align*}
\end{claim}

Take $\eta,\gamma>0$ and $0<\eps<u<T$ to be fixed later and let $\delta(u,\eta,\gamma)< \gamma$ be given as in Claim \ref{c.epsiloncoupling}.
We now use the iterative construction to build a coupling between $\X^{(T)}$ and $\X^{(T,\eps)}$ inductively on the intervals $[t_n^\ep,t_{n+1}^\ep]$, where $t_n^\ep$ and denotes the times arising in Algorithm \ref{a.C} (we will similarly use $t_n$ for those arising in Algorithm \ref{a.B}).
We will say that step $n$ of the coupling is successful if $\X^{(T)}$ and $\X^{(T,\eps)}$ belong to $\Omega_N^\gamma$ and $\|\X^{(T)}(t_n^\eps)-\X^{(T,\eps)}(t_n^\eps)\|\leq \delta$.

We begin by fixing an initial condition $\X(0)\in\Omega_N$ and setting $\X^{(T)}(0)=\X^{(T,\ep)}(0)=\X(0)$.
With this, step 0 of the coupling is clearly successful.
Now suppose that the we have built the coupling successfully up to time $t_n^\ep$.
Since $\X^{(T)}$ is a time-homogenous Markov process, we can restart its construction at time $t_n^\ep$, starting from $\X^{(T)}(t_n^\eps)$.
We use the coupling from Claim \ref{c.epsiloncoupling} to build copies $\wcX^{(T,n)}$ and $\wcX^{(T,n)}_\eps$ of $\wcX^{(T)}$ with initial conditions $\X^{(T)}(t_{n}^\eps)$ and $\X^{(T,\eps)}(t_{n}^\eps)$, and use them to perform the $n$-th steps of each of Algorithms \ref{a.B} and \ref{a.C}.
Since $u<T$, the two trajectories coincide after time $T$, and this implies that $t_{n+1}\leq t_{n+1}^\eps<t_{n+1}+\eps$.
Then we let $\X^{(T)}$ evolve independently from time $t_{n+1}$ to time $t_{n+1}^\eps$.
If the coupling fails at this step (i.e. if either $\X^{(T)}(t_n^\eps)$ or $\X^{(T)}(t_n^\eps)$ do not belong to $\Omega_N^\gamma$ or $\|\X^{(T)}(t_n^\eps)-\X^{(T,\eps)}(t_n^\eps)\|>\delta$), we stop the construction and we continue the coupling independently.

Fix $R>0$ and let $n(R)$ be the first time such that $t_n>R$. 
We want to estimate the probability that $\X^{(T)}$ and $\X^{(T,\ep)}$ are at distance bigger than $r>0$ on the interval $[0,R]$.
To that end we write
\begin{align*}
    &\pp\Big(\sup_{s\in [0,R]}\|\X^{(T)}(s) - \X^{(T,\eps)}(s)\| \geq r \Big)  \\
    &\;\;\;\leq\pp(n(R)>M)+ M\eta+ \pp(t_{n+1}-t_n^\eps\leq u, \ \forall n \leq M)+\pp\!\left( \X^{(T)}(s)\notin \Omega_N^{\gamma}\text{ for some }s\in [0,R]\right )\\
    &\qquad\;+ \pp\bigg(\sup_{n\leq M} \sup_{\substack{s_1,s_2\in [0,R]\\|s_1-s_2|\leq u}} \|\wcX^{(T,n)}_\eps(s_1)-\wcX^{(T,n)}_\eps(s_2)\|\geq \frac{r}{2} , \sup_{\substack{s_1,s_2\in [0,R]\\|s_1-s_2|\leq u}} \|\X^{(T)}(s_1)-\X^{(T)}(s_2)\|\geq \frac{r}{2}\bigg) \\
    &\quad+ \pp\bigg ( \sup_{\substack{s_1,s_2\in [0,R]\\|s_1-s_2|\leq \eps}} \|\X^{(T)}(s_1)-\X^{(T)}(s_2)\|\geq \delta(u,\eta,\gamma)\bigg ).
\end{align*}
Take $\upsilon>0$. We first choose an $M>0$ such that $\pp(n(R)>M) < \upsilon$. 
We then choose $\eta$, $u$ and $\gamma$ so that the sum of the remaining terms in the second line are smaller than $3 \upsilon$. 
We may ask $u$ to be even smaller so that third line also becomes smaller than $\upsilon$. 
Having chosen $\eta,u,\gamma>0$, as well as $M$, we may now choose $\eps$ to be small enough so that the last term is also smaller than $\upsilon$. We get a bound of $6\upsilon$ for the probability we are interested, and since $\upsilon$ is arbitrary, the conclusion follows.

It remains then to show that the claim holds.

\begin{proof}[Proof of Claim \ref{c.epsiloncoupling}]
Let us first couple $\X_{x_0}$ with $\X_{y_0}$. We will couple each component independently, so we   assume that $N=1$ and $x_0<y_0 \in \rr$. The one-dimensional coupling is as follows: start running $\X_{x_0}$ and $\X_{y_0}$ independently until the first time they intersect, $\tau\coloneqq\inf\{s>0\!: \X_{x_0}(s)=\X_{y_0}(s)\}$. After that, continue the two trajectories together, i.e. make  $\X_{x_0}\big|_{[\tau,\infty)} = \X_{y_0}\big|_{[\tau,\infty)}$.
We claim that for any $\eta_0>0$ there is a $\delta>0$ such that $\pp(\tau<u)\leq \eta_0$ if $y_0-x_0\leq\delta$: in fact, if $\tilde \X_{x_0}$ and $\tilde \X_{y_0}$ are  independent solutions of \eqref{e.SDE} starting from $x_0$ and $y_0$, we have (for some standard Brownian motion $B_t$)
\begin{align}
    \tilde \X_{y_0}(t)-\tilde \X_{x_0}(t)= y_0-x_0+\sqrt{2}B_{t}+ \int_0^{t} v'(\tilde \X_{y_0}(s))-v'(\tilde \X_{x_0}(s)) ds \leq \delta + \sqrt{2} B_t + 2\|v'\|_\infty t.
\end{align}
Letting $Z_t$ denote the process on the right hand side, we have $\pp(\tau>u)\leq\pp(\inf_{s\in [0,u]} Z(s)>0)$, which converges to $0$ as $\delta \to 0$.

The above argument implies that we can couple $\X_{x_0}$ with $\X_{y_0}$ when both live in $\rr^N$ in such a way that $\pp\big(\X_{x_0}\big|_{[u,\infty)}=\X_{y_0}\big|_{[u,\infty)}\big)>1-\eta/2$ if $\|y_0-x_0\|\leq\delta$.

We now construct $\wcX_{x_0}$ and $\wcX_{y_0}$ using acceptance-rejection sampling as follows: sample i.i.d. copies $(\X_{x_0}^j, \X_{y_0}^j)_{j\in\nn}$ of the coupling prescribed above, define
\begin{align*}
j_{x_0}=\inf\{j\!:\,\X_{x_0}^j(t)\;\;\forall\,t\in[0,T] \in \Omega_N \} \qqand j_{y_0}=\inf\{j\!:\,\X_{y_0}^j(t)\;\;\forall\,t\in[0,T] \in \Omega_N \},
\end{align*}
and set $\wcX^{(T)}_{x_0}=\X^{j_{x_0}}_{x_0}$ and $\wcX^{(T)}_{y_0}=\X^{j_{y_0}}_{y_0}$.
The now claim follows from the construction if we show that, with probability converging to $1$ as $\delta \to 0$, $j_{x_0}=j_{y_0}$ if $\|y_0-x_0\|\leq\delta$.
But this is a consequence of the facts that
\begin{align*}
\inf _{x_0\in \Omega_N^{\gamma}} \pp\big(\X_{x_0}(s)\in \Omega_N\;\;\forall s \in [0,T]\big) >0
\end{align*}
and that 
\begin{align*}
\sup_{x\in \Omega_N^{\gamma}} \pp\big(\X_{x}(s)\notin \Omega_N\;\uptext{for some}\;s\in [0,u']\big) \xrightarrow[u'\to 0]{} 0,
\end{align*}
as $v'$ is bounded.
\end{proof}
\end{proof}

\section{From myopic NIBM to myopic NIRW}\label{sec:fromNIBMtoNRW}

The goal of this section is to prove Theorem \ref{t.main}.
Before doing so, we need to understand the behavior of a single solution of \eqref{e.SDE} as $\kappa$ gets large.

\subsection{Metastability for Brownian motion in a periodic potential}\label{sec:F-W}

Let $\X^\kappa$ denote a single Brownian particle $\X^\kappa$ subject to the periodic potential $\kappa v$, i.e., a solution solution of \eqref{e.SDE} with the given choice of $\kappa$.
The goal of this section is to study the behavior of $\caja{\X^\kappa}$ as $\kappa\to\infty$ where, we recall, $\caja{\X^{(T^\kappa)}}$ denotes the process that tracks the integer sites visited by each coordinate (see the notation \eqref{e.caja} and \eqref{e.caja-vec}).
We want to show that if $\X^\kappa$ is a solution of \eqref{e.SDE} then, under the correct time reparametrization, $\caja{\X^\kappa}$ converges to a Poisson process (as stated in \eqref{e.FW}).

The argument will be based on the classical metastability results for Brownian motion on a double-well potential in the regime of the Friedlin-Wentzell large deviation theory for random pertubations of dynamical systems \cite{freidlinWentzell}.
The basic idea is that, when observed on an interval which is an integer translate of  $[-1/2+\eta,3/2-\eta]$ (for small $\eta$), the potential to which $\X^\kappa$ is subject looks, for large $\kappa$, like a double-well potential, i.e., a smooth potential $g(x)$ with two local minima (at $x=0$ and $x=1$) and one local maximum (at $x=1/2$) and which goes to infinity as $x\to\pm\infty$.
The behavior of a Brownian particle subject to the potential $g$ can then be used to understand the evolution of $\caja{\X^\kappa}$ in the appropriate regime.

We begin then by describing these classical results, in a slightly more general setting.
The results which we state first appeared, as far as we know, in \cite{Eulalia_paper}, but in our presentation we follow Chapter 5.2 of \cite{Eulalia_book}.
Compared to that book, we state the results under a change of time so that we are under a strong drift rather than a weak diffusion regime\footnote{To be more precise, we speed up time by $\kappa=\epsilon^{-2}$ in their notation.}.

\begin{figure}
    \centering
    \includegraphics[width=3in]{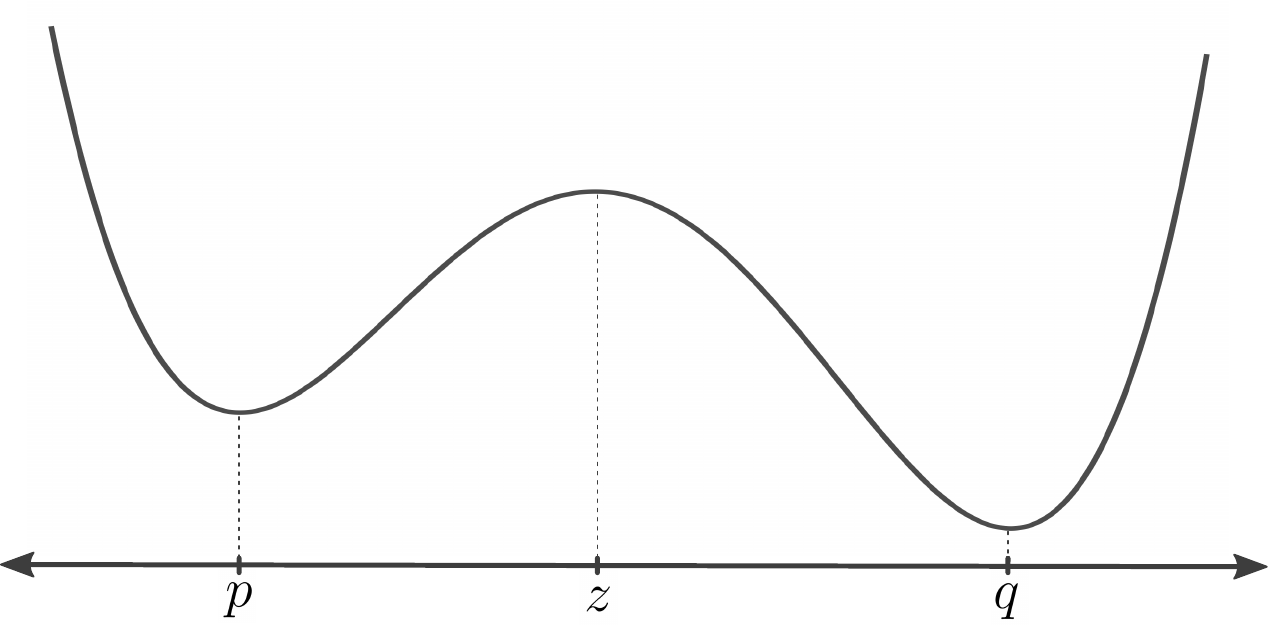}
    \caption{A double-well potential as in Section \ref{sec:F-W} (in Section \ref{s.pf-main}, where the metastability results are applied, we will use $p=0$, $z=1/2$, and $q=1$).}
    \label{f.v-potential}
\end{figure}

Take $\kappa>0$ and a function $g\in C^2(\rr)$. 
Assume that $\X^\kappa$ solves the equation
\begin{equation}
    \X^\kappa(t) = -\kappa\int_0^t g'(\X^\kappa(t))dt + dB_t,\label{e.sde-3}
\end{equation}
where $B$ is a standard Brownian motion.
In this subsection we assume that the following hypotheses hold:
\begin{enumerate}[label=(H\arabic{*})]
\item $g$ goes to infinity faster than linearly, i.e., $\lim_{|x|\to \infty} g(x)/|x| = \infty$.
\item $g$ has exactly three critical points, $p<z<q$, with  $g(q)<g(p)<g(z)$. 
Furthermore, $g''(p), g''(q)>0$ and $g''(z)<0$ (see Figure \ref{f.v-potential}). 
\item $g'$ is globally Lipschitz, that is, there is $K>0$ such that for any $x$, $y \in \rr$ 
\begin{align*}
|g'(x)-g'(y)| \leq K|x-y|.
\end{align*}
\end{enumerate}
We will refer to this set of assumptions as \emph{hypothesis (H)}.

Now we fix an arbitrary height $\hat g>g(z)$ and define 
\[\tau^\kappa=\inf\{t>0\!:\X^\kappa_t=q\;\text{or}\;g(\X^\kappa)=\hat g\}.\]
We are interested in the stopping time $\tau^\kappa$ with the process starting to the left of $z$, in which case it should be thought of as the first time that the process hits either $q$ to the right of $z$ or some arbitrary large height $\hat g$ to the left of $z$ (see again Figure \ref{f.v-potential}).
Note that our assumptions on $b$ ensure that both hitting times are finite almost surely.

Let $\lambda^\kappa$
be the unique positive real numbers such that
\begin{align*}
\pp_p(\tau^\kappa>\lambda^\kappa)=e^{-1}.
\end{align*}
The following result gives precise information on the size of $\lambda^\kappa$ and on the distribution of $\tau^\kappa$ and $\X^\kappa_{\tau^\kappa}$ for large $\kappa$:

\begin{thm}\label{t.beta}
    Assume that $\X^\kappa$ satisfies hypothesis (H) and has initial condition $x<z$.
    Then
    \begin{align}
      \label{e.Delta}&\lim_{\kappa \to \infty} \frac{\log \lambda^\kappa }{\kappa } = 2(g(z)-g(p))\eqqcolon\Delta
    \end{align}
    Furthermore, $\tau^\kappa/\lambda^\kappa$ converges in law to an exponential random variable with parameter 1, and $\X^\kappa_{\tau^\kappa} = q$ with probability tending to $1$ as $\kappa\to\infty$.
\end{thm}

\begin{proof}
    This is the one dimensional version of \cite[Thm 5.5]{Eulalia_book}, with two main differences.
    The first one is that we are sending the size of the drift to infinity instead of sending the strength of the the noise to zero, but it is straightforward to see using Brownian scaling that this does not change the result.
    The main difference is that their definition of $\tau^\kappa$ is slightly different: it is the first time the process comes within a certain distance of $q$. 
    But it is easy to see\footnote{By, for example, a coupling argument.} that, since we are in dimension 1, the results holds in exactly the same way if one considers instead the hitting time of $q$.
    We also included the possibility that $g(\X^\kappa_t)=\hat g$ in our definition of the hitting time $\lambda^\kappa$.
    This does not appear in the definition of the hitting time in the statement of the cited theorem, but it does appear in the same way in its proof. One can also see that $\X^\kappa_{\tau^\kappa} = q$ with probability tending to $1$ as $\kappa\to\infty$ by a simple coupling argument.
\end{proof}

We can also estimate how much time $\X^\kappa$ spends close to $p$ before the stopping time $\tau^\kappa$:

\begin{thm}\label{t.typically_close_p}
    Under the same assumptions of Theorem \ref{t.beta}, we have that for any $\delta>0$ and any $\alpha<\Delta$,
    \begin{align*}
    \lim_{\kappa \to \infty}\pp_x\left(\sup_{ 0\leq s \leq \tau^\kappa-2e^{\alpha\kappa}} \uptext{Leb}\left (\left\{ u\in [s\lambda^\kappa, s\lambda^\kappa + e^{\alpha\kappa}]: \X^\kappa_u \notin B(p,\delta)\right\} \right) >\delta e^{\alpha \kappa} \right )=0.
    \end{align*}
\end{thm}

\begin{proof}
    This is a direct consequence of \cite[Thm. 5.6]{Eulalia_book}.
\end{proof}

Now we go back to our setting of a periodic potential, and state the precise convergence result for $\caja{\X^\kappa}\ts$.

\begin{prop}\label{p.FW}
    Let $\X^{\kappa}$ be a solution of \eqref{e.SDE} with parameter $\kappa$ starting from $x\in \zz+[-1/4,1/4]$ and let $\Y$ be a Poisson process with rate $1$.
     Then, there exists a family of constants
     \[\lambda^\kappa = \exp(2\kappa + o(\kappa))\]
    such that $\caja{\X^\kappa}(\lambda^\kappa \cdot)\longrightarrow\Y(\cdot)$ in distribution for the Skorohod topology on $D[0,\infty)$.
\end{prop}

Before turning to the proof, it will be useful to state a simple condition in our setting for convergence of Markov chains in the standard Skorohod topology.
Let $(X_t)_{t\geq 0}$ be a càdlàg continuous time Markov process in $\zz^d$ with jumps which are such that $\|X_t-X^-_t\|=1$. 
Let $0<\sigma_1<\sigma_2<\dotsm$ denote the jumping times of $X$, let $\tau_i=\sigma_{i}-\sigma_{i-1}>0$ be the time elapsed between those jumps (here $\sigma_0=0$), and write $x_i=X_{\sigma_i}-X_{\sigma_i}^-$ for the value of the $i$-th jump.
Note that in order to recover $X$ it suffices to know the sequence $(\tau_i,x_i)_{i\in \nn}$.

\begin{lem}\label{l.increments_enough_Skorohod}
    Let $(X^n_t)_{t\geq 0,n\in \nn}$ be a sequence of càdlàg continuous time Markov processes in $\zz^N$ with jumps which also satisfy $\|X^n_t-(X^n)^-_t\|=1$. 
    Assume that the associated sequence $(\tau^n_i,x^n_i)_{i\in \nn}$ converges to $(\tau_i,x_i)_{i\in \nn}$ in the product topology, and assume furthermore that
    \begin{align*}
        \textstyle\sum_{i\in\nn} \tau_i =\infty\ \  \text{ and  } \ \ \tau_i >0 \ \ \ \forall i\in \nn.
    \end{align*}
    Then $X^n$ converges to $X$ in the Skorohod topology on $D[0,\infty)$, where $X$ is the càdlàg process associated to $(\tau_i,\sigma_i)_{i\in \nn}$.
\end{lem}

The proof of this is standard, so we omit it.

\begin{proof}[Proof of Proposition \ref{p.FW}]
   Without loss of generality we may assume that $x\in [-1/4,1/4]$.
   Fix a constant $\eta>0$ so that $\hat g\coloneqq v(-1/2+\eta)>v(1/2)$ (this is possible by continuity since $v(-1/2)>v(0)$) and consider a function $g$ which is equal to $v$ in $[-1/2+\eta, 1+\eta]$ and is such that it satisfies hypothesis (H) with $p=0$, $q=1$ and $z=1/2$ (the latter is also clearly possible: the definition of $\eta$ ensures the right ordering for the heights at the critical points while, for example, we can let $g$ be defined as suitable convex parabolas on $(-\infty,-1/2]$ and $[1+2\eta,\infty)$, with $C^2$ monotone interpolations in $(-1/2,-1/2+\eta)$ and $(1+\eta,1+2\eta)$).
   Let $\wt\X^\kappa$ be a solution of \eqref{e.sde-3}.
   We choose the constants $\lambda^\kappa$ to be the ones which appear in Theorem \ref{t.beta} for this diffusion.
   Note that here $\Delta=2(g(1/2)-g(0))=2$, so $\lambda^\kappa=e^{2\kappa+o(\kappa)}$ as prescribed.
   
   Now let $\tau^\kappa_{x}$ be the hitting time of $\X^\kappa$ to $x$, and define $\tilde\tau^\kappa_x$ similarly in terms of $\wt\X^\kappa$.
   Let $\tau_{x,y}=\tau_{x}\wedge \tau_{y}$ and $\tilde\tau_{x,y}=\tilde\tau_{x}\wedge \tilde\tau_{y}$.
   Thanks to Lemma \ref{l.increments_enough_Skorohod}, it is enough to show that that $\tau_{-1,1}/\lambda^\kappa$ converges to an exponential random variable with parameter $1$ and that $\pp(\tau_{-1,1} = \tau_{-1})\xrightarrow[\kappa\to\infty]{}0$.
   To this end we couple $\X^\kappa$ and $\wt \X^\kappa$ starting from the same point and using the same Brownian motion. 
   From Theorem \ref{t.beta} we know that $\tilde \tau^\kappa_{-1,1}/\lambda^\kappa$ converges to an exponential random variable with parameter $1$, and that $\pp(\tilde \tau^\kappa_{-1,1}=\tilde \tau^\kappa_{1})\xrightarrow[\kappa\to\infty]{}0$.
   From this we get the desired conditions, because the probability that $\tau_{-1,1}$ is different from $\tilde \tau_{-1,1}$ is bounded by
   \begin{align*}
    \pp\big(\wt\X^\kappa (\tilde \tau^\kappa_{-1/2+\eta, 1})=-1/2+\eta\big),
   \end{align*}
   which also converges $0$ also thanks to Theorem \ref{t.beta}. 
\end{proof}

\begin{rem}
    Note that the result in Proposition \ref{p.FW} still holds for finitely many independent copies of $(\X^\kappa_i)_{i=1}^N$ starting at different $x_i$'s.
    This is true thanks to Lemma \ref{l.increments_enough_Skorohod} and the fact that the processes $(\Y_i)_{i=1}^N$ almost surely never jump at the same time.
\end{rem}

\subsection{Proof of Theorem \ref{t.main}}\label{s.pf-main}

In order to prove Theorem \ref{t.main}, the basic idea will be to couple the algorithms used to build the mBM and the mRW by using systems of independent copies of the solution of \eqref{e.SDE} and of independent Poisson processes, which are in turn coupled so that they remain close to each other.
These two systems can be coupled in such a way thanks to Proposition \ref{p.FW}. 

The first thing which we need to prove is that different copies of $\X^\kappa$, started at sufficiently close times and positions, remain close in distribution.
We do this in the next result, based again on coupling.

Before stating the result we note that, since $v \in C^\infty$ and it has a minimum at $0$, we know that there exists an $\eta>0$ such that $v''(y)>0$ for all $y \in [-2\eta,2\eta]$ and such that $v(-1/2+ \eta)>1$.
We will employ this choice of $\eta$ throughout the rest of this section.

\begin{lem}\label{l.coupling_different_initial_conditions}
    Let $x, \bar x \in \rr^N$ be such that $x_i,\bar x_i \in \zz + (-1/4,1/4)$ and $|x_i-\bar x_i|\leq 1/2$ for each $i$.
    Assume that $\X^\kappa$, respectively $\bar\X^\kappa$, are $N$ independent solutions of \eqref{e.SDE} with initial values $x$, respectively $\bar x$.
    Then there exists a function $g\!:\rr^+\times \rr^+ \longrightarrow [0,1]$ and a coupling of $(\X^\kappa, \bar\X^\kappa)$ depending on $x$, $\bar x$ and a parameter $h>0$, whose law we denote by $\pp_{x,\bar x,h}$, so that for all $0<\alpha<2$ and all $\delta>0$,
    \begin{equation}\label{e.inf-coupl}
    \inf_{|h|\leq \delta}\inf_{x,\bar x}\,\pp_{x,\bar x,h}( \X^\kappa(t)= \bar\X^\kappa(t+h \lambda^\kappa) \text{ for all } t> e^{\alpha\kappa})\geq g(\kappa,\delta)
    \end{equation}
    and
    \begin{align}\label{e.function_g}\lim_{\delta \to 0} \limsup_{\kappa \to \infty} g(\kappa,\delta)=\lim_{\kappa \to \infty} g(\kappa,e^{-(2-\alpha)\kappa})=1.\end{align}
\end{lem}

\begin{proof}
Since the coordinates of $\X^\kappa$ and $\bar\X^\kappa$ are independent, it suffices to show the result for $N=1$, in which case we may assume that $-1/4\leq \bar x\leq  x\leq 1/4$. 
It is enough to find a coupling between $\X^\kappa(\cdot)$ and $\bar\X^\kappa(\cdot + h \lambda^\kappa)$ such that the two diffusions intersect at a time smaller than $e^{\alpha\kappa}$ with high probability (and uniformly in $h,x,\bar x$), since we can evolve them together afterwards.

Consider first the case $h=\delta = 0$ and $x,\bar x \in [-\eta,\eta]$, with $\bar x < x$ and $\eta$ as specified before the statement of the lemma.
We will use the coupling
\begin{align}\label{e.coupling_independiente}
\d\X^\kappa(t)= -\kappa v'(\X^\kappa(t)) \d t + \d B_t,\qquad
\d\bar \X^\kappa(t)= -\kappa v'(\bar \X^\kappa(t)) \d t + \d\bar B_t,
\end{align}
where $B$ and $\bar B$ are two independent Brownian motions\footnote{One can (slightly) improve the speed of convergence by choosing the coupling $(B,\bar B)$ more carefully, but the independent coupling is sufficient for our purpuses and will, in fact, be useful also later.}.
Define $\tau$ to be the first time when either $\X^\kappa$ or $\bar\X^\kappa$ exits $[-2\eta,2\eta]$ , 
let $U_t= \X^\kappa(t\wedge\tau) - \bar\X^\kappa(t\wedge\tau)$, and observe that (by our choice of $\eta$)
\begin{align*}
U_t \leq x-\bar x+ \sqrt{2}W_{t\wedge \tau},
\end{align*}
where $W$ is a Brownian motion.
Then, under this coupling, we have
\begin{align*}
  \pp(\X^\kappa(t)=\bar\X^\kappa(t)&\;\uptext{for some}\;0\leq t\leq e^{\alpha\kappa})=\pp(U_s=0\;\text{for some}\;t\in [0,e^{\alpha \kappa}])\\
  &\geq \pp\big (W_t =\tfrac{1}{\sqrt{2}}(\bar x - x) \;  \text{ for some }0<t<e^{\kappa(\alpha\wedge\beta) }\big) - \pp\big (\tau< e^{\kappa\beta}\big)
\end{align*}
for any $\beta>0$.
The first term on the right hand side clearly converges to $1$ as $\kappa\nearrow \infty$, and it does so uniformly in $x,\bar x\in[-\eta,\eta]$.
And if we choose $\beta<2v(\eta)$ then the second term goes to $0$, also uniformly in $x,\bar x\in[-\eta,\eta]$, thanks to the same argument as the one used in the proof of Proposition \ref{p.FW}: we couple $\X^\kappa$ with a diffusion with a drift $\tilde v$ that equals $v$ in $[-\eta,\eta]$ and satisfies hypothesis (H), and conclude by applying Theorem \ref{t.beta}.
From this, we obtain that, uniformly in $x,\bar x\in[-\eta,\eta]$,
\begin{equation}\label{e.coupling_uniform_starting_same_time}
  \pp(\X^\kappa(t)=\bar\X^\kappa(t)\;\uptext{for some}\;0\leq t\leq e^{\alpha\kappa})\xrightarrow[\kappa\to\infty]{}1.
\end{equation}

Now we turn to the general case.
By symmetry it is enough compute the infimum in \eqref{e.inf-coupl} over $h\geq0$, so we want construct a coupling where $\X^\kappa$ has started $0<h\lambda^\kappa\leq\delta \lambda^\kappa$ units of time before $\bar\X^\kappa$, and where $\X^\kappa$ and $\bar\X^\kappa$ now start at points $x,\bar x\in[-1/4,1/4]$.
Take $t\in [0,h\lambda^\kappa]$, couple $\X^\kappa$ and $\bar\X^\kappa$ independently, and define a new stopping time $\sigma$ to be the first time $t$ such that $\X^\kappa(t)$ and $\bar\X^\kappa(t + h\lambda^\kappa)$ both belong to $[-\eta,\eta]$. 
We will show that $\pp(\sigma\geq \frac12e^{\alpha \kappa} ) $ is upper bounded by some function $\tilde g(\kappa,\delta)$ satisfying \eqref{e.function_g}.
The result then follows (after adjusting $\alpha$ slightly) by using the coupling introduced in the previous case to run $\X^\kappa(t)$ and $\bar\X^\kappa(t + h\lambda^\kappa)$ for $t\geq\sigma$. 

In order to obtain the desired estimate for  $\pp(\sigma\geq \frac{1 }{2 }e^{\alpha \kappa})$, we note first that it is bounded by
\begin{multline*}
\pp\!\left (\X^\kappa(t) \notin [-\tfrac 12 +\eta, 1]\;\text{or}\;
\bar\X^\kappa(t) \notin [-\tfrac 12 +\eta , 1]\;\text{for some}\;t\in [0,4h \lambda^\kappa]\right ) \\
+ \pp\!\left (\sigma \geq \tfrac{1 }{2 }e^{\alpha \kappa}\;\uptext{and}\; \X^\kappa(t), \bar\X^\kappa(t)\in [-\tfrac 12 + \eta, 1]\;\uptext{for all}\;t\in [0,4h \lambda^\kappa] \right ).
\end{multline*}
For large enough $\kappa$, the first term is uniformly bounded by $3(1- e^{-2\delta})$, again by a simple comparison argument with the Friedlin-Wentzell regime for a double-well potential and Theorem \ref{t.beta}. 
The second term is upper bounded by the probability that in the time interval $[0,4h\lambda^\kappa]$ either $\X^\kappa$ or $\bar\X^\kappa(\cdot + h\lambda^\kappa)$ spends a proportion of time bigger than $1/2$ outside the interval $[-\eta,\eta]$  while not exiting $[-1/2 + \eta,1]$. This is uniformly small thanks to Theorem \ref{t.typically_close_p}.
\end{proof}

\begin{rem}\label{r.special_coupling}
    Lemma \ref{l.coupling_different_initial_conditions} as stated above will be a key tool for what follows. 
    However, at the end of the proof of Theorem \ref{t.main} we will need a slightly improved version of the coupling, in the case of equal starting times, which will allow us to construct, in the same probability space, copies $\hat\X^{\kappa,x}$ of the process $\X^\kappa$ started at all initial conditions $x\in [-1/4,1/4]$ together with the process $\X^\kappa$ started at $0$.
    The coupling goes as follows.
    Define first $\X^\kappa$ to be a solution of \eqref{e.SDE} started at $0$ using a Brownian motion $B$. 
    Consider a second a Brownian motion $\dot B$, which is independent of $B$, and use it to construct solutions $\dot \X^{\kappa,x}$ of \eqref{e.SDE} started at every $x\in[-1/4,1/4]$.
    Define $\tau^x$ as the first time $t$ when $\dot \X^{\kappa,x}_t= \X^{\kappa}_t$ and define a process  $\hat \X^{\kappa,x}(t)$ to be $\dot \X^{\kappa x}(t)$ for $t\leq\tau^x$ and $\X^{\kappa}(t)$ for $t>\tau^x$. 
    Then $\hat\X^{\kappa,x}$ indeed solves \eqref{e.SDE} for $v$, starting at $x$, and the estimate \eqref{e.inf-coupl} (with $\delta=0$) applies in the same way for each pair $(\X^\kappa,\hat{\X}^{\kappa,x})$.
\end{rem}

We will need the following lemma, which states that $\X^\kappa$ and $\caja{\X^\kappa}$ exit the Weyl chamber at close enough times with high probability.

\begin{lem}\label{l.both_hit}
Let $\X^\kappa= (\X_i^\kappa)_{i=1}^N$ be a system of $N$ independent solutions of \eqref{e.SDE} with a given choice of $\kappa$ and with initial conditions $x^{\kappa}(0) \in \Omega_N+[-1/4,1/4]^N$.
Define the stopping times
\begin{align*}
\tau^\kappa= \inf\{t>0: \X^\kappa \notin \Omega_N\}\qqand\caja{\tau^\kappa}= \inf\{t>0: \caja{\X^\kappa} \notin \Omega_N\}.
\end{align*}
Then for any $\alpha>0$ we have
\begin{align*}
\lim_{\kappa\nearrow \infty} \pp_x\!\left(\big|\ttsm\caja{\tau^\kappa}-\tau^\kappa\big|\geq e^{\alpha\kappa}\right)  = 0.
\end{align*}
\end{lem}

\begin{proof}
Without loss of generality, we may prove this result by restricting to $N=2$ and $\alpha<1$. 
We first study the event $\{\tau^\kappa>\caja{\tau^\kappa}+ e^{\alpha\kappa}\}$.
Its occurrence implies that $[\X_1^\kappa](\caja{\tau^\kappa})=[\X_2^\kappa](\caja{\tau^\kappa})$, but $\X^\kappa$ does not intersect in the time interval $\big[\tsm\caja{\tau^\kappa}, \caja{\tau^\kappa} +e^{\alpha \kappa}\big]$.
We now proceed as in the proof of Lemma \ref{l.coupling_different_initial_conditions}, defining $\sigma$ to be the first time after $\tau^\kappa$ when $\X^\kappa_1$ and $\X^\kappa_2$ are in the same integer translate of $[-\eta,\eta]$. Note that the probability that $\sigma$ is greater than $\frac12e^{\alpha \kappa}$ is at most
\begin{align*}
&\pp(\caja{\X_1^\kappa} \text{ or } \caja{\X_2^\kappa}\;\text{jump in}\;(\caja{\tau^\kappa}, \caja{\tau^\kappa} +e^{\alpha\kappa}]) + \pp(\sigma>\tfrac12e^{\alpha \kappa}\;\text{and no jumps in}\;(\caja{\tau^\kappa}, \caja{\tau^\kappa} +e^{\alpha\kappa}])).
\end{align*}
The first term here goes to $0$ thanks to Proposition \ref{p.FW} and the remark that follows it.
The second term goes to $0$ by the same reasoning as in Lemma \ref{l.coupling_different_initial_conditions}, namely we couple with Brownian motions in a double-well potential and use Theorem \ref{t.typically_close_p}.
Hence $\sigma\leq\frac12e^{\alpha\kappa}$ with high probability, and then we may use \eqref{e.coupling_uniform_starting_same_time} starting at time $\sigma$ to show that $\X_1^\kappa$ and $\X_2^\kappa$ are equal with high probability by time $e^{\alpha\kappa}$.

Next we study the event $\{\caja{\tau^\kappa}>\tau^\kappa + e^{\alpha\kappa}\}$, which now means that $\X_1^\kappa(\tau^\kappa) = \X_2^\kappa(\tau^\kappa)$, but $[\X^\kappa]$ does not intersect in $[\tau^\kappa, \tau^\kappa+e^{\alpha \kappa}]$.
We first need to show that, with high probability, $\X_1^\kappa(\tau^\kappa)$ determines the integer that $\caja{\X_1^\kappa}$ the first hits after time $\tau^\kappa$.
To this end, for $\varepsilon>0$ define the bad set
\begin{align*}
B_\varepsilon^\kappa=\big\{x\in \rr\!: \pp_x( \X^\kappa_1(\tau_\zz)=k \text{ and } \tau_\zz \leq e^{\alpha \kappa})>\varepsilon\;\text{for}\; k\in\{\lfloor x\rfloor,\lfloor x\rfloor+1\}\big\},
\end{align*}
where $\tau_\zz$ is the first time $\X_1^\kappa$ hits the integers; these are points near local maxima of the potential $v$ which have a fair chance to first hit either of the two integers to its side.
Note that we already know that $\pp_x(\tau_{\zz} >e^{\alpha \kappa})$ is small for large $\kappa$ thanks to Theorem \ref{t.typically_close_p}, so the condition $\tau_\zz \leq e^{\alpha \kappa}$ is redundant, but it makes the proof clearer.

We claim that, with high probability as $\kappa \to \infty$, $\X_1^\kappa(\tau^\kappa) \notin B_\varepsilon^\kappa$. 
Assume that this is not the case, i.e. that there is a sequence $\kappa_n\nearrow\infty$ so that $\bar\ep\coloneqq\lim_{n\to\infty}\pp(\X_1^{\kappa_n}(\tau^{\kappa_n}) \in B_\varepsilon^{\kappa_n})>0$.
Writing $\kappa$ in place of $\kappa_n$ for simplicity, we restart $\X_1^\kappa$ and $\X_2^\kappa$ at time $\tau^\kappa$, when they both take the value $x\coloneqq\X_1^\kappa(\tau^\kappa)$.
Observe that for large $\kappa$ the two processes touch a unique integer between times $\tau^\kappa$ and $\tau^\kappa + e^{\alpha \kappa}$ with probability at least $1-\varepsilon$ by Theorem \ref{t.typically_close_p}. 
And if $x\in B^\kappa_\ep$, then $\caja{\X}_1^\kappa(\tau^\kappa+e^{\alpha \kappa})$ and $\caja{\X}_2^\kappa(\tau^\kappa+e^{\alpha \kappa})$ can each be either $\lfloor x\rfloor$ or $\lfloor x\rfloor+1$, independently, each option with probability at least $\varepsilon$. 
But this contradicts the fact that, with high probability, $(\caja{\X^\kappa_1}, \caja{\X^\kappa_2})$ is close to $(\Y_1,\Y_2)$ for the Skorohod topology because, with probability at least $\bar\ep(1-\ep)\varepsilon^2$, both $\caja{\X^\kappa_1}$ and $\caja{\X^\kappa_2}$ have a jump in the interval $[\tau^\kappa, \tau^\kappa + e^{\alpha \kappa} ]$.

Hence, for fixed $\ep>0$ we have $\pp(\X_1^\kappa(\tau^\kappa) \in B_\varepsilon^\kappa)<\ep$ for large enough $\kappa$.
And on the event $\{\X_1^\kappa(\tau^\kappa) \notin B_\varepsilon^\kappa\}$, and for large enough $\kappa$, the first integer that each of $\caja{\X^\kappa_1}$ and $\caja{\X^\kappa_2}$ hits after time $\tau^\kappa$ is the same with probability at least $1-2\ep$.
Again by Theorem \ref{t.typically_close_p} the two processes hit a unique integer within the time interval $[\tau^\kappa,\tau^\kappa+e^{\alpha\kappa}]$ with probability at least $1-\ep$ for large $\kappa$.
So $\pp(\caja{\tau^\kappa}>\tau^\kappa + e^{\alpha\kappa})<4\varepsilon$, for large $\kappa$, and the result follows.
\end{proof}

Next we need to show that, at typical times, each $\X_i^\kappa$ is close to an integer.

\begin{lem}\label{l.typical_time_close_to_integers}
    In the context of Lemma \ref{l.both_hit}, we have for any $\alpha<2$ and $\beta <1$ that
    \begin{align*}
        \lim_{\kappa \to \infty}\sup_{t>e^{\alpha \kappa}} \pp_x\Big(\X^\kappa_i(t)\notin \zz+[-\beta,\beta ]\;\uptext{for some $i\in\{1,\dotsc,N\}$}\Big) =0.
    \end{align*}
\end{lem}

\begin{proof}
    Without loss of generality we can assume that $N=1$.
    We begin by noting that, thanks to Theorem \ref{t.typically_close_p} and a simple coupling argument similar to the one in previous proofs, we have that for any given $\ep>0$ and any $M>0$ (which may depend on $\kappa$),    \begin{align}\label{e.average_time}
    \frac{1}{ e^{\alpha \kappa}}\ee\!\left[\int_{Me^{\alpha\kappa} }^{(M+1) e^{\alpha\kappa}}\mathbf 1_{\X_1^\kappa(s)\notin \zz+[-\beta,\beta ]} ds\right] \leq \varepsilon,
    \end{align}
    for large enough $\kappa>0$.

    Now, arguing by contradiction, consider an arbitrary $\ep>0$ and suppose that there is a deterministic $t>e^{\alpha \kappa}$ such that 
    \begin{equation}\label{e.detnotin0}
    \pp_x\big(\X^\kappa_1(t)\notin \zz+[-\beta,\beta ]\big)>2\varepsilon.
    \end{equation}
    Then by the coupling starting at different initial times given in Lemma \ref{l.coupling_different_initial_conditions}, we have that for all $s\in(t,t+e^{\kappa \alpha})$
    \begin{equation}\label{e.detnotin}
      \pp_x\big(\X^\kappa_1(s)\notin \zz+[-\beta,\beta ]\big)>\varepsilon.
    \end{equation}   
    This contradicts \eqref{e.average_time}, finishing the proof.
\end{proof}

We now have all the tools necessary to prove Theorem \ref{t.main}.

\begin{proof}[Proof of Theorem \ref{t.main}]
The proof is based again on a coupling argument.
Let $\X$ denote $N$ independent copies of the solution of \eqref{e.SDE} started at $x\in \Omega_N^\zz + [-1/4,1/4]$, with the given choice of $\kappa$.
By Proposition \ref{p.FW}, we can couple $\X(\lambda^\kappa \cdot)$ with $\Y(\cdot)$, a system of $N$ independent Poisson processes in $\zz$ jumping at rate $1$, started at $\caja{x}\in \Omega_N$, in such a way that the Skorohod distance between $\caja{\X}(\lambda^\kappa \cdot)$ and $\Y$ restricted to $[0,L]$ is small.

Recall the definition of the processes $\wcY ^{(L)}$ and $\wcX^{(T^\kappa)}$ in Sections \ref{s.direct_mRW} and \ref{s.glueing} (i.e. the systems of diffusions and random walks conditioned, respectively, on non-intersecting within $[0,T^\kappa]$ and $[0,L]$).
We will first show that $\cajab{\wcX^{(T^\kappa)}}(\lambda^\kappa\cdot)$ can be coupled with $\wcY^{(L)} (\cdot)$ in such a way that their Skorohod distance restricted to $[0,2L]$ is small.
This can be achieved by sampling the two processes together using coupling mentioned in the previous paragraph and then employing a plain acceptance-rejection   algorithm on both.
More precisely, we fix $\delta>0$, sample i.i.d. pairs $(\X_j(\lambda^\kappa \cdot),\Y_j)_{j\in \rr}$, each coupled as above, and define
\begin{align*}
    J^\kappa&= \inf\{j\in \nn: \X_j(t) \in \Omega_N \;\uptext{ for all }\; t\in [0,T^\kappa] \},\\
    J&= \inf \{j \in \nn: \Y_j(t) \in \Omega_N \;\uptext{ for all }\; t\in [0,L] \},\\
    E&= \inf\{j\in \nn: d((\caja{\X_j}(\lambda^\kappa t))_{t\in [0,2L]}, (\Y_j(t) )_{t\in [0,2L]})> \delta \},
\end{align*}
where $d$ denotes the Skorohod distance.
The acceptance-rejection method tells us that $\X_{J^\kappa}$ has the law of $\wcX^{(T^\kappa)}(\lambda_k\cdot)$ in $[0,T^\kappa/\lambda^\kappa]$ and that $\Y_{J}$ has the law of $\wcY ^{(L)} (\cdot)$ restricted to $[0,L]$.
Note that $J<\infty$ almost surely.

We claim now that 
\begin{equation}\label{e.JJE}
  \pp(J^\kappa=J<E)\xrightarrow[\kappa\to\infty]{}1.
\end{equation}
That $J<E$ with high probability is simple: $J$ is a geometric random variable with positive, fixed parameter, while $E$ is a geometric random variable with parameter converging to $0$ as $\kappa \nearrow \infty$ under the coupling, so indeed $\pp(J<E)\longrightarrow1$ as $\kappa\to \infty$. 

Next we check that it is unlikely that $J<J^\kappa$.
This event can occur in three ways: either  $\caja{\X_J}$ and $\Y_J$ are far, or they are close but $\Y_J$ has jumps occurring too close to time $L$, or the exit times from $\Omega_N$ of $\X_J$ and $\caja{\X_J}$ are very different.
That is,
\begin{align*}
\pp(J<J^\kappa)\leq \pp(J<E) + \pp(\Y_J \text{  jumps in }[L-\delta,L+\delta]) + \pp\!\left(|\caja{\tau^\kappa_J}-\tau^\kappa_J|\geq e^{\alpha\kappa}\right).
\end{align*}
As $\kappa\to\infty$ the right hand side goes to $\pp(\Y_J \text{  jumps in }[L-\delta,L+\delta])$ by the previous argument and Lemma \ref{l.both_hit}, which now goes to $0$ as $\delta\to0$ (where we use that $J<\infty$ almost surely to handle the dependence on $J$).

To finish proving \eqref{e.JJE} we need to check that it is also unlikely that $J>J^\kappa$.
The argument is the same as in the previous case, with the only difference that in the upper bound as $\kappa\to\infty$ we now get $\pp(\Y_{J^\kappa} \text{  jumps in }[L-\delta,L+\delta])$.
The only difficulty is that $J^\kappa$ now depends on $\kappa$, but we just proved that $J^\kappa\leq J$ with high probability, so we can argue as above.

Now, using this coupling of $\wcY^{(L)}$ and $\wcX^{(T^\kappa)}$, we run one iteration of step (2) of Algorithms \ref{a.A} and \ref{a.B} respectively; we will use superscripts $\Y$ and $\X$ to distinguish the sequences of times $\tau_n$ and $t_n$ employed in each algorithm.
It will be convenient for us to stop the algorithms if $\tau_1^\Y-L>L/2$, respectively $\tau_1^\X-T^\kappa>T^\kappa/2$ (meaning that we replace $t_1^\Y=\tau_1^\Y-L$ by $t_1^Y\wedge L/2$ and $t_1^X=\tau^\X_1-L$ by $t_1^\X\wedge T^\kappa/2)$.
This can be done because we have already shown that the mRW and the mBM are Markov processes.
Now we claim that for any $\alpha<2$, with probability going to $1$ as $\kappa\to \infty $ and then $\delta\to 0$,
\begin{enumerate}
    \item $|t_1^\X/\lambda^\kappa - t_1^\Y| \leq \delta $. 
    \item $\wcX^{(T^\kappa)}(t_1^\X)\in \Omega_N^\zz+ [-1/4,1/4]^N$.
    \item The Skorohod distance between $\cajab{\wcX^{(T^\kappa)}(\lambda^\kappa\cdot)}\big|_{[0,L/2]}$ and $ \wcY^{(L)}\big|_{[0,L/2]}$ is less than $\delta$.
    \item $\wcY^{(L)}$ does not jump in $[t_1^{\Y}-\delta,t_1^{\Y}+\delta]$. 
    Thus, and as long as (3) holds, $\cajab{\wcX^{(T^\kappa)}(\lambda^\kappa\cdot)} \big|_{[0,t_1^{\Y}]}$ and $\cajab{\wcX^{(T^\kappa)}(\lambda^\kappa\cdot)} \big|_{[0,t_1^{\X}/\lambda^\kappa]}$, are within Skorohod distance $\delta$, respectively, from $ \wcY^{(L)}\big|_{[0,t_1^{\Y}]} $ and $\wcY^{(L)}\big|_{[0,t_1^{\X}/\lambda^\kappa]}$.
\end{enumerate}
Conditions (1) and (2) ensure that we can restart the iteration of Algorithms \ref{a.A} and  \ref{a.B} using the same coupling, while conditions (3) and (4) ensure that the coupling which we have built up to time $t^\Y_1$ keeps the processes close in Skorohod distance.

That (1) occurs with high probability as $\kappa\to\infty$ follows directly from Lemma \ref{l.both_hit} and Proposition \ref{p.FW}, while (3) occurs with high probability thanks to the above construction (and \eqref{e.JJE} in particular).

Consider now condition (4).
Note first that, by the Markov property, conditionally on $\wcY^{(L)}(L)$ the law of $\wcY^{(L)}\big|_{[0, L/2]}$ is independent from that of $\wcY^{(L)}\big|_{[L,\infty)}$. 
We also have that for all $\varepsilon>0$ there is a finite set $K\subseteq\zz^N$ such that $\pp\big(\wcY^{(L)}(L)\in K\big)>1-\varepsilon$ and $\inf_{z\in K}\pp\big(\wcY^{(L)}(L)=z\big)>0$ for all $z\in K$.
Since $t_1^{\Y}$ is a deterministic function of $\wcY^{(L)}\big|_{[L,\infty)}$, we have that 
\begin{align*}
&\pp\!\left(\wcY^{(L)}\text{ jumps in }[t_1^\Y-\delta, t_1^\Y+\delta]\right)
=\ee\!\left[  \pp\!\left (\wcY^{(L)} \text{ jumps in }[t_1^\Y-\delta, t_1^\Y+\delta]\,\middle|\, \wcY^{(L)}(L), \wcY^{(L)}\big|_{[L,\infty)} \right )\right ]\\
&\hspace{0.1\textwidth}\leq \ee\!\left[\sup_{t\in [0,L/2]} \pp\!\left( \wcY^{(L)} \text{ jumps in }[t-\delta, t+\delta] \,\middle|\, \wcY^{(L)}(L)\right ) \right ]\\
&\hspace{0.1\textwidth} \leq \varepsilon + \frac{|K| }{\displaystyle\inf_{z\in K}\pp(\wcY^{(L)}(L)=z) } \sup_{t\in [0,L/2]} \pp\!\left(\wcY^{(L)} \text{ jumps in }[t-\delta, t+\delta]\right)\xrightarrow[\delta \to 0]{} \varepsilon.
\end{align*}
As $\varepsilon$ is arbitrary, (4) is satisfied with high probability. 

Let us pause for a moment to discuss why we chose this approach to prove (4), as we will require a similar---though necessarily more intricate---strategy in the continuous case to prove (2).
The desired property clearly holds if $t_1^{\Y}$ is replaced by a deterministic time $t$, so a natural strategy for the proof would be to show that the first intersection time $\tau_1^\Y$  is almost independent of the initial path of the process.
However, pursuing this directly would require controlling the law of the Markovian bridge that arises when the terminal value $\wcY^{(L)}(L)$ is fixed, which entails additional technical work. 
To avoid this, we use the fact that, given a fixed initial condition, with high probability $\wcY^{(L)}(L)$ takes values in some (large) finite set, so this Markovian bridge is absolutely continuous with respect to the law of the process itself on $[0,L]$. 
One might expect this technique to break down in the case of $\wcX^{(T^\kappa)}$, as the process takes uncountably many values at time $T^\kappa$ (so absolute continuity between the bridge and the process cannot be established in the same way).
We will resolve this by showing that, in essence, all that matters is the collection of boxes in which  $\wcX^{(T^\kappa)}(T^\kappa)$ lives.
We turn to this next.

The main idea to prove (2) is to condition on $\cajab{\wcX^{(T^\kappa)}(T^\kappa)}$, the vector of closest integers to the entries of $\wcX^{(T^\kappa)}(T^\kappa)$ (not to be confused with $\cajab{\wcX^{(T^\kappa)}}(T^\kappa)$).
The price we pay is that we no longer have the conditional independence which we used in the previous case, but we can get around that by coupling $\wcX^{(T^\kappa)}$ starting at time $T^\kappa$ with a copy of it which starts at $\cajab{\wcX^{(T^\kappa)}(T^\kappa)}$, for which the desired conditional independence will hold, and then showing that the coupling can be chosen so that the two copies of the process are close.

To implement this strategy, start by fixing $\ep>0$ and noting that, with probability larger than $1-\ep$, $\cajab{\wcX^{(T^\kappa)}}(T^\kappa)=\cajab{\wcX^{(T^\kappa)}(T^\kappa)}$ for large enough $\kappa$.
Now we define the coupling.
We first sample $\wcX^{(T^\kappa)}$ restricted to $[0,T^\kappa]$.
Then, to sample its continuation after time $T^\kappa$, we use the coupling introduced in Remark \ref{r.special_coupling} (in each coordinate separately) to construct a continuation from every possible value of $\wcX^{(T^\kappa)}(T^\kappa)$ and call $\mathsf{W}$ the solution started from $\cajab{\wcX^{(T^\kappa)}(T^\kappa)}$. 
Define $t^\mathsf{W}_1$ as the first time that $\mathsf{W}$ exits the Weyl chamber, and note that thanks to Lemma \ref{l.typical_time_close_to_integers} and the construction in Remark \ref{r.special_coupling}, $\pp(t^\mathsf{W}_1= t^\X_1)>1-\ep$.
The construction ensures that $t^\mathsf{W}_1$ and $\wcX|_{[0,T^{\kappa}/2]}$ are indeed conditionally independent given $\cajab{\wcX^{(T^\kappa)}(T^\kappa)}$, so we have
\begin{align*}
&\pp\left( \wcX^{(T^\kappa)}(t_1^\X)\notin \Omega_N^\zz+ [-\tfrac14,\tfrac14]^N\right )\\
&\hspace{0.06\textwidth}\leq 2\eps + \pp\!\left(\wcX^{(T^\kappa)}(t^\mathsf{W}_1)\notin \Omega_N^\zz+ [-\tfrac{1 }{4 },\tfrac{1 }{4 }]^N ,\,t^\mathsf{W}_1=t_1^{\X},\,\cajab{\wcX^{(T^\kappa)}(T^\kappa)}= \cajab{\wcX^{(T^\kappa)}}(T^\kappa) \right)  \\
&\hspace{0.06\textwidth}= 2\eps + \ee\!\left[\pp\!\left(\wcX^{(T^\kappa)}(t^\mathsf{W}_1)\notin \Omega_N^\zz+ [-\tfrac{1 }{4 },\tfrac{1 }{4 }]^N \,\middle|\,\cajab{\wcX^{(T^\kappa)}(T^\kappa)} \right)  \right ] \\
&\hspace{0.06\textwidth}\leq 3\ep + \frac{|K|}{\inf_{z\in K} \pp\big(\cajab{\wcX^{(T^\kappa)}(T^\kappa)} =z \big) } \sup_{t\in [0,T^\kappa/2]}\pp\!\left(\wcX^{(T^\kappa)}(t)\notin \Omega_N^\zz+ [-\tfrac{1 }{4 },\tfrac{1 }{4 }]^N\right ) \xrightarrow[\kappa \to \infty]{} 3\ep,
\end{align*}
where we have chosen as before a finite set $K\subseteq \zz^N$ so that $\pp\left(\wcX^{(T^\kappa)}(T^\kappa)\in K\right )>1-\ep$ and $\inf_{z\in K} \pp\big(\cajab{\wcX^{(T^\kappa)}(T^\kappa)} =z \big)>0$. 

We conclude by iterating the algorithm until we hit a given time $M>0$.
The fact that conditions (1)-(4) are met implies that on an event of high probability we can continue iterating Algorithms \ref{a.A} and \ref{a.B} while keeping the created processes within Skorohod distance $\delta$.
When we run a new iteration of the algorithm, say passing from step 1 to step 2, there is a minor issue which we need to be careful with: $\wcX^{(T^\kappa)}$ and $\wcY^{(L)}$ have been constructed up to two different times ($t^\X_1/\lambda^\kappa$ and $t^\Y_1$) and, at those times, they are at different points.
In order to solve this we couple $\wcY^{(L)}$ starting at time $t^\Y_1$ with a copy of $\wcX^{(T^\kappa)}$ which also starts at time $t^\Y_1$ and at $\wcY^{(L)}(t^\Y_1)$, for which the construction works.
Lemma \ref{l.coupling_different_initial_conditions} and the fact that neither process jumps between times $t_1^\X$ and $t_1^\Y$ with high probability ensure that this copy of $\wcX^{(T^\kappa)}$ is at Skorohod distance at most $\delta$ from the one we are interested in (starting at time $t^\X_1$ and at the correct location) with high probability.
This modification allows us to keep the coupling going, with the constructed processes staying within distance $\delta$ with high probability.
At each iteration the construction advances a time which is equal to independent copies of the random variable $t_1^\Y$.
Since $t_1^\Y$ is stochastically lower bounded by (the minimum between $L/2$ and) an exponential random variable with parameter $N$, the construction finishes in a finite number of steps, and the result follows.
\end{proof}

\vs

\noindent{\bf Acknowledgements.}
The authors thank Santiago Saglietti for discussions and references about metastability for Brownian motion on a double well potential. 
They also thank Antonio Rojo and Hernán Torres for producing videos with simulations of myopic random walks \cite{simulation}. 
The three authors were supported by Centro de Modelamiento Matem\'{a}tico Basal Funds FB210005 from ANID-Chile.
DR was also supported by Fondecyt Grant 1241974, and AS was also supported by Fondecyt Grants 1240884 and 11200085, and by ERC 101043450 Vortex.
AS would also like to thank the Hausdorff Institute of Mathematics, and in particular the trimester program ``Probabilistic methods in quantum field theory'', where he was based while the final part of this work was written.

\printbibliography[heading=apa]

\end{document}